\documentclass[11pt,notitlepage,a4paper,reqno,ps]{amsart}
\usepackage[utf8]{inputenc}
\usepackage{color}
\usepackage{xcolor}

\usepackage{amsmath,amsfonts,euscript,amssymb,amsthm,a4,times}
\usepackage[colorlinks,pdfpagelabels,pdfstartview = FitH,bookmarksopen = true,bookmarksnumbered = true,linkcolor = blue,plainpages = false,hypertexnames = false,citecolor = red, pagebackref=true]{hyperref}

\allowdisplaybreaks

\def\XXint#1#2#3{{\setbox0=\hbox{$#1{#2#3}{\int}$}
\vcenter{\hbox{$#2#3$}}\kern-.5\wd0}}

\newtheorem{theorem}{Theorem}[section]
\newtheorem{corollary}[theorem]{Corollary}
\newtheorem{lemma}[theorem]{Lemma}

\theoremstyle{definition}
\newtheorem{definition}[theorem]{Definition}

\makeatletter \catcode`@=11
\newbox\tr@tto
\setbox\tr@tto=\hbox{{\count0=0\dimen0=-,9pt\dimen1=1,1pt\loop\ifnum
    \count0<11 \advance \count0 by1 \vrule width.51pt height\dimen1
    depth\dimen0\kern-0.17pt\advance\dimen0 by-0.05pt\advance\dimen1
    by0.1pt\repeat \loop\ifnum\count0<21\advance \count0 by1 \vrule
    width.6pt height\dimen1 depth\dimen0\kern-0.2pt \advance\dimen0
    by-0.1pt\advance\dimen1 by 0.05pt\repeat}}
\def\medint{\displaystyle\copy\tr@tto\kern-10.4pt\int}
\numberwithin{equation}{section}

\def\R{{\mathbb R}}
\def\RN{{\mathbb R}^{N}}

\def\proofof#1{\begin{proof}[Proof of #1]}

\def\loc{\rm loc}

\def\R{{\mathbb R}}
\def\RN{{\mathbb R}^{N}}

\def\proofof#1{\begin{proof}[Proof of #1]}

\def\loc{\rm loc}

\def\uj{u_{j}}

\catcode`@=11
\newbox\tr@tto
\setbox\tr@tto=\hbox{{\count0=0\dimen0=-,9pt\dimen1=1,1pt\loop\ifnum\count0<11 \advance \count0 by1 \vrule width.51pt height\dimen1
     depth\dimen0\kern-0.17pt\advance\dimen0 by-0.05pt\advance\dimen1
     by0.1pt\repeat \loop\ifnum\count0<21\advance \count0 by1 \vrule
     width.6pt height\dimen1 depth\dimen0\kern-0.2pt \advance\dimen0
     by-0.1pt\advance\dimen1 by 0.05pt\repeat}}
\def\medint{\displaystyle\copy\tr@tto\kern-10.4pt\int}
\catcode`@=12

\newcommand{\LL}{\mathrm{L}}
\newcommand{\WW}{\mathrm{W}}

\newcommand{\CC}{\mathrm{C}}

\numberwithin{equation}{section}

\begin{document}
\title[Lipschitz regularity for a priori bounded minimizers]{Lipschitz regularity for a priori bounded minimizers of integral functionals with nonstandard growth}
\author[M. Eleuteri -- A. Passarelli di Napoli]{Michela Eleuteri -- Antonia Passarelli di Napoli}
\address{Dipartimento di Scienze Fisiche, Informatiche e Matematiche, Universit\`a degli Studi di Modena E Reggio Emilia, via Campi 213/b, 41125 Modena, Italy}

\email{michela.eleuteri@unimore.it}

\address{Dipartimento di Matematica e Applicazioni ``R. Caccioppoli''
\\
Universit\`a degli Studi di Napoli ``Federico II''\\
Via Cintia, 80126, Napoli (Italy)}

\email{antpassa@unina.it}

\maketitle

\begin{abstract}
We establish the Lipschitz regularity of the a priori bounded local minimizers of integral functionals with non autonomous energy densities satisfying non standard growth conditions under a sharp bound on the gap between the growth and the ellipticity exponent.
\end{abstract}



\bigskip

\section{Introduction}

Since the pioneering papers by P. Marcellini \cite{M89, M91}, the Lipschitz regularity for minimizers of integral functionals with non-standard growth and for weak solutions for the associated Dirichlet problem to the elliptic
system has attracted a lot of attention (see e.g. \cite{BM, BS2022, CMMPdN, DFM2021, EMM, EMM20, EMMP, EPdN23, M2021, M2022}).
\\
One of the main motivations comes from the applications, for instance to the theory of elasticity for strongly anisotropic materials (see Zhikov \cite{Z}, and also \cite{ZkO}); to this aim, in recent years the integral of the Calculus of Variations
\begin{equation}
\label{funz_modello}
\int_{\Omega} |Du|^p + a(x) |Du|^q \, dx,
\end{equation}
where the function $a = a(x)$ is  H\"older continuous with exponent $\alpha$ and where $1 < p < q,$ has been widely investigated from the point of view of the regularity of local minimizers. In particular M. Colombo and G. Mingione, (\cite{colmin}), studied the regularity of minimizers for integrals of the type \eqref{funz_modello} under the sharp gap
\begin{equation}
\label{cond_Mingione}
\frac{q}{p} < 1 + \frac{\alpha}{n}.
\end{equation}
On the other hand, M. Eleuteri, P. Marcellini and E. Mascolo (\cite{EMM}) investigated more general integrals of the Calculus of Variations of the type
\begin{equation}
\label{funzionaleconlag}
F(u) = \int_{\Omega} g(x, |Du|) \, dx
\end{equation}
where
\begin{equation}
\label{Hpdistruttura}
g(x, |Du|) = |Du|^p + a(x) \, |Du|^q
\end{equation}
is just a model example, without therefore assuming the precise structure condition for the integrand as in \eqref{funz_modello}; they proved the local Lipschitz continuity of the local minimizers and to the  solutions to the corresponding elliptic systems assuming a $\WW^{1,r}$ regularity on the coefficients and under the gap 
\begin{equation}\label{ipotesiconpqr}
\frac{q}{p} < 1 + \frac{1}{n} - \frac{1}{r}.
\end{equation}
\noindent In the model case  \eqref{Hpdistruttura}, the above condition \eqref{ipotesiconpqr} is equivalent to \eqref{cond_Mingione} by the Sobolev embedding with
\begin{equation}
\label{sobolev}
\alpha = 1 - \frac{n}{r}.
\end{equation}
On the other hand, it is well known that, when dealing with a priori bounded minimizers of functionals with  non standard growth, the regularity can be obtained under a bound on the gap independent of the dimension $n$ (\cite{Adimurthy, BF05, CKP, Choe, CMM14, ELM99, GenGioTor, L94}), see also \cite{HO} in the case of functionals with quasi isotropic $(p,q)-$growth. In particular, for the double phase functional, in \cite{colmin2}, 
the authors were able to prove that the a priori bounded local minimizers of integral functionals of kind \eqref{funz_modello} are $\mathcal{C}^{1,\beta}$-regular provided the sharp bound 
\begin{equation}
\label{sharp-a-priori-bound}
q \le p + \alpha
\end{equation}
holds.
\\
It is natural to ask if  the same phenomenon persists when the Lipschitz regularity of more general functionals of kind \eqref{funzionaleconlag} is investigated under an analogous a priori sharp bound on the gap between the exponents $p$ and $q.$
The main motivation comes from the fact that there are several interesting examples of functionals with non-standard growth and with Uhlenbeck structure that are not covered by the double-phase functional \eqref{funz_modello} or Orlicz-type functionals such as $g(t) = t^p \log (1 + t)$; for instance we refer to Remark 3.3 in \cite{BMS18} where an example of an integrand function exhibiting $p,q-$growth but not satisfying a $\Delta_2-$condition is presented. 
\\
Our paper aims to answer this open question, by studying the local Lipschitz continuity of the a priori bounded solutions to a class of variational  problems of the form
 \begin{equation}
\label{defint}
\min_{z \in \WW^{1,p}_{\loc}(\Omega; \mathbb{R}^N)}\int_\Omega F(x, Dz)\,dx,
\end{equation}
where $\Omega$ is a bounded open set of $\mathbb{R}^n$, $n \ge 2$.

\noindent We shall consider integrands $F$ such that $\xi \mapsto F(x, \xi)$ is $\mathcal{C}^2$ and   there exists $f:\Omega\times \mathbb{R}^{nN} \mapsto [0,+\infty)$ such that $F(x, \xi)=f(x,|\xi|)$. Such an assumption simplifies the approximation procedure that, even in the scalar case, can be quite involved (see for instance \cite{EMM20}). 
\\
We shall assume the following set of
conditions:

$$  \ell {(1 + |\xi|^2)^{\frac{p}{2}}} \le \, F(x,\xi) \le \, L {(1 + |\xi|^2)^{\frac{q}{2}}} \eqno{\mathrm{(F1)}}$$
$${\nu} (1 + |\xi|^2)^{\frac{p-2}{2}} |\lambda|^2 \le \, \sum_{i, \ell, \alpha, \beta} F_{\xi_i^{\alpha} \xi_{\ell}^{\beta}} (x,\xi) \lambda_i^{\alpha} \lambda_{\ell}^{\beta} \eqno{\mathrm{(F2)}}$$
$$ |F_{\xi_i^{\alpha}\xi_{\ell}^{\beta}}(x,\xi)|
\le \, \tilde{L} (1 + |\xi|^2)^{\frac{q-2}{2}} \eqno{\mathrm{(F3)}}$$
$$ |{F_{x \xi}}(x, \xi) | \le \, h(x) {(1 + |\xi|^2)}^{\frac{q-1}{2}} \eqno{\mathrm{(F4)}}$$
for almost all $x \in \Omega$,  and all $\xi, \lambda \in \mathbb{R}^{nN}$, $\xi = \xi_i^{\alpha}, \lambda = \lambda_{{\ell}}^{\beta},$ $i, \ell = 1, \dots, n,$ $\alpha, \beta = 1, \dots, N$, where $ 2\le p\le q$ and $0 \le {\nu} \le \tilde{L}$  are fixed constants,  and $h(x)\in \LL^r_{\mathrm{loc}}(\Omega)$ is a fixed non
negative function.
\\
Before stating our main result, we recall the definition of local minimizer
\begin{definition}
A mapping $u\in \WW^{1,1}_{\rm loc}(\Omega, \RN )$ is a local minimizer {of the integral functional \eqref{defint}} if
$F(x,Du) \in \LL^{1}_{\rm loc}(\Omega )$ and
\begin{equation}\label{minineq}
 \int_{\mathrm{supp}\varphi} \! F(x,Du) \, dx \leq \int_{\mathrm{supp}\varphi} \! F(x,Du+D\varphi) \, dx   
\end{equation}

for  any $\varphi\in {C}_0^{\infty}(\Omega, \RN )$.
\end{definition}

The main result reads as follows.

\begin{theorem}
\label{main}
Let $u\in \LL^{\infty}_{\mathrm{loc}}(\Omega; \mathbb{R}^N)\cap \WW^{1,p}_{\mathrm{loc}}(\Omega; \mathbb{R}^N)$ be a local minimizer of the functional \eqref{defint} under
the assumptions {\rm {(F1)--(F4)}}. Assume moreover that
\begin{equation}\label{erre}
   r>\max\{n,p+2\}  
\end{equation}
and
\begin{equation}\label{gap}
  q<p+1-\max\left\{\frac{n}r,\frac{p+2}{r} \right\}.
\end{equation} 
Then $u$ is locally Lipschitz continuous and the following estimate holds for any ball $B_{R_0} \Subset \Omega$
\[
{||Du||_{\LL^\infty \left(B_{\frac{R_0}{2}}; \mathbb{R}^{nN}\right)}\le  C\left( 1+||u||_{\LL^{\infty}(B_{R_0}; \mathbb{R}^N}) \right)^{\hat{\chi}}},
\]
{with $C \equiv C (n,N,  \nu, \tilde L, ||h||_{L^r(\Omega)}, R_0)$ and with a positive exponent $\hat{\chi}=\hat{\chi}(p,q,r,n)$.}
\end{theorem}

We observe that condition \eqref{gap} not only reduces to \eqref{sharp-a-priori-bound} under the Sobolev embedding with \eqref{sobolev} for $p < n - 2,$ but also includes the case $p < n < p + 2:$ indeed the a priori higher integrability $\LL^{p+2}$ reveals to be crucial in order to weaken the assumption on the coefficients  
 in the non-autonomous case.

The proof of this result goes along several steps. The first step is devoted to the construction of the approximating problems in Section \ref{tre.due} based on the approximation lemma stated in Section \ref{tre.uno}; the main feature here is that the approximating local minimizers have norm in a suitable Lebesgue space which is uniformly bounded by the $\LL^{\infty}$ norm of the local minimizer $u$. This procedure, inspired by \cite{CKP} and already used in a similar form in \cite{GPdN}, is one of the main and delicate points of our arguments. Indeed, in the general vectorial setting, the a priori boundedness of the minimizer of the original functional does not imply the boundedness of the approximating minimizers. However,  this construction  complicates the form of the integrand function of the approximating functionals and, despite they satisfy
standard growth conditions with respect to the gradient variable, the growth with respect to the $u$ variable in our energy density yields the necessity to establish the Lipschitz regularity of the approximating minimizers in Section \ref{tre.tre}; the proof of this result relies on a classical Moser iteration argument and makes use of a preliminary higher differentiability and higher integrability result proven in \cite{GPdN}.
The next step 
aims to prove, in Section \ref{tre.quattro},  a second order Caccioppoli type inequality
for the approximating minimizers; the main point
here is that we are going to establish it with constants independent of the approximation parameters. In a further step, in Section \ref{quattro}, by using a  Gagliardo-Nirenberg type inequality (\cite{CKP}), we establish a uniform higher integrability result for the approximating minimizers, with constants independent of the parameter of the approximation. Finally we are ready to prove in Section \ref{cinque} the main result of the paper, that will be divided in
two steps. In the first one we establish an uniform a priori estimate for the $\LL^{\infty}$ norm
of the gradient of the minimizers of the approximating functionals while, in the second, we show that these estimates are preserved in passing to the limit.

We conclude by mentioning that, as a consequence of the Lipschitz regularity of the local minimizers, we are also able to obtain a second order regularity result. More precisely, we have the following:

\begin{theorem}
\label{highdiff}
Let $u\in \LL^{\infty}_{\mathrm{loc}}(\Omega; \mathbb{R}^N)\cap \WW^{1,p}_{\mathrm{loc}}(\Omega; \mathbb{R}^N)$ be a local minimizer of the functional \eqref{defint} under
the assumptions {\rm {(F1)--(F4)}}. Assume moreover that
\eqref{erre}
   and \eqref{gap}
  are in force.
Then $u\in W^{2,2}_{\mathrm{loc}}(\Omega; \mathbb{R}^N)$  and the following estimate holds for any ball $B_{R_0} \Subset \Omega$
\[
{||D^2u||_{\LL^2 \left(B_{\frac{R_0}{2}}; \mathbb{R}^{n^2N}\right)}\le  C\left( 1+||u||_{\LL^{\infty}(B_{R_0}; \mathbb{R}^N}) \right)^{\hat{\chi}}}
\]
{with $C \equiv C (n,N,  \nu, \tilde L, ||h||_{L^r(\Omega)}, R_0)$ and with a positive exponent $\hat{\chi}=\hat{\chi}(p,q,r,n)$.}
\end{theorem}

\section{Preliminary}
In what follows, we shall denote by $C$  a
general positive constant that may vary on different occasions, even within the
same line of estimates.
Relevant dependencies  will be suitably emphasized using
parentheses or subscripts.  The symbol $B(x,r)=B_r(x)=\{y\in \R^n:\,\, |y-x|<r\}$ will denote the ball centered at $x$ of radius $r$.
\\
We recall the following well known iteration lemma, whose
proof can be found, e.g. in \cite[Lemma 6.1, p.191]{Giusti}.

  \begin{lemma}\label{lem:Giaq}
    For $0<R_1<R_2$, consider a bounded function
    $f:[R_1,R_2]\to[0,\infty)$ with
    \begin{equation*}
      f(r_1)\le\vartheta f(r_2)+\frac A{(r_2-r_1)^\alpha}+ \frac B{(r_2-r_1)^\beta}+ C
       \qquad\mbox{for all }R_1<r_1<r_2<R_2,
    \end{equation*}
    where $A,B,C$, and $\alpha,\beta$ denote nonnegative constants
    and $\vartheta\in(0,1)$. Then we have
    \begin{equation*}
      f(R_1)\le c(\alpha,\vartheta)
      \bigg(\frac A{(R_2-R_1)^\alpha}+\frac B{(R_2-R_1)^\beta}+C
      \bigg).
    \end{equation*}
  \end{lemma}

\section{The approximation}

\label{tre}

\subsection{An Approximation Lemma}

\label{tre.uno}

In this subsection we will state a Lemma that will be the main tool in the approximation procedure. For the proof we refer to Proposition 4.1 in \cite{CupGuiMas} ( see also   \cite[Lemma 4.1]{CGGP} and the recent  \cite[Theorem 5.1]{DeFilLeo}).

 \begin{lemma}
\label{apprcupguimas2parte}
  Let $F:\Omega\times\R^{nN}\to [0,+\infty)$   be a
Carath\'eodory function 
satisfying assumptions \textnormal{(F1)--(F4)}. Then there exists a sequence
 of Carath\'eodory functions
$F^{j}:\Omega\times\R^{nN}\to [0,+\infty)$, monotonically convergent to $F$,
such that the following properties hold for a.e. $x\in \Omega$ and for every  $\xi
      \in
  \R^{nN}$:
    \begin{equation}
         \label{(II)}
 F^{j}(x,\xi)\le F^{j+1}(x,\xi)\le F(x,\xi)\qquad \forall  j\in \mathbb{N}
     \end{equation}
   \begin{equation}
  \label{p-p-1approx}
\begin{cases}
 K_{0}(|\xi|^{p}-1)\le F^{j}(x,\xi)\le L(1+|\xi|)^{q}\\
  F^{j}(x,\xi)\le K_{1}(j)(1+|\xi|)^{p},   
\end{cases} 
\end{equation}
with positive constants  $K_{0}=K_{0}(\ell)$ and
 $K_{1}(j)$.
In addition for every  $\xi\in
  \R^{nN}$, there hold 
\begin{equation}
   \label{(I)} F^{j}(x,\xi)={\tilde
 F}^{j}(x,|\xi|),  \qquad t\mapsto {\tilde
 F}^{j}(x,t) \ \ \text{ nondecreasing},
  \end{equation}
  \begin{equation}
  \label{(III)}
   \sum_{i, \ell, \alpha, \beta} F^j_{\xi_i^{\alpha} \xi_{\ell}^{\beta}} (x,\xi)\lambda_i^{\alpha} \lambda_{\ell}^{\beta} \ge {\overline \nu}(1+|\xi|^2)^{\frac{p-2}{2}}|\lambda|^2\qquad \forall \lambda, {\xi}\in
 \R^{nN},\end{equation}with  ${\overline \nu}=\overline \nu(\nu,p)>0$.
 We also have
 \begin{equation}
 \label{(VII)}
\begin{cases}
|F^j_{\xi \xi}(x,\xi)|\le C(j)(1+|\xi|^2)^{\frac{p-2}{2}}
\\
|F^j_{\xi \xi}(x,\xi)|\le C(\tilde L)(1+|\xi|^2)^{\frac{q-2}{2}}.    
\end{cases}
\end{equation}
Moreover,
 the vector field $x \mapsto F^{j}_{\xi}(x,\xi)$ is
weakly
 differentiable and,  for  every   $\xi\in
  \R^{nN}$,
 \begin{equation}
 \label{(V)}
\begin{cases}
   |F^j_{x \xi}(x,\xi)|\le C(j)h(x)(1+|\xi|^2)^{\frac{p-1}{2}}
\\
|F^j_{x \xi}(x,\xi)|\le Ch(x)(1+|\xi|^2)^{\frac{q-1}{2}}. 
\end{cases}
\end{equation}

\end{lemma}

\subsection{The approximating problems}

\label{tre.due}

\noindent Here we present the construction of the approximating problems that is inspired by the one in  \cite{CKP} and whose main feature is that the sequence of the approximating minimizers has norm in a suitable Lebesgue space uniformly bounded by the $\LL^\infty$ norm of the minimizer $u$.

\medskip
Fix a compact set  $\Omega'\Subset \Omega$ and    a real number $a\ge ||u||_{\LL^\infty(\Omega^\prime; \mathbb{R}^N)}$. For $m\in \mathbb{N},$ 
 let ${\uj} \in  \WW^{1,p}(\Omega^{\prime}; \mathbb{R}^N)\cap \LL^{2m}(\Omega^{\prime}; \mathbb{R}^N)$
be a minimizer to the functional
\begin{equation}\label{fun}
\mathfrak{F}^{j}(v,\Omega^{\prime})= \int_{\Omega^{\prime}} \! \Bigl( F^j(x,Dv) + \bigl( |v|^2-a^2 \bigr)^{m}_{+}  \Bigr)\,dx
\end{equation}
 under the boundary condition
 $$ \uj=u \qquad \text{on}	\qquad \partial\Omega^\prime,$$
 and where $F^j$ is the sequence of functions obtained applying Lemma \ref{apprcupguimas2parte} to the integrand $F$ of the functional at \eqref{defint}.
 \\
The existence of $\uj$ is easily established by the direct methods of the Calculus of Variation. 
We shall need the following

\begin{lemma}\label{lem62}
As $j \to +\infty$, we have that 
$$
\int_{\Omega^{\prime}} \! \Bigl( \bigl( |\uj |-a \bigr)^{2m}_{+}
\Bigr) \, dx \to 0, \quad  \quad
\int_{\Omega^{\prime}} \! F^j(x,D\uj ) \, dx\to \int_{\Omega^{\prime}} \! F(x,Du) \, dx.
$$
and
\begin{center}
$D\uj \to Du \quad$ strongly in \;\;\; $\LL^p(\Omega^{\prime}; \mathbb{R}^{nN})$.
\end{center}
\end{lemma}

\begin{proof}
By the minimality of $\uj$, using $u$ as test function in the minimality inequality at \eqref{minineq}, we get

\begin{eqnarray}\label{conv1}
\int_{\Omega^{\prime}} \! \Bigl( F^j(x,D{\uj} ) + \bigl( |{\uj} |^2-a^2 \bigr)^{m}_{+} \Bigr)\,dx
\leq \int_{\Omega^{\prime}} \!  F^j(x,Du )\,dx, 
\end{eqnarray}
since $|u|\le  ||u||_{\LL^\infty(\Omega^\prime; \mathbb{R}^N)}\le a$ a.e. in $\Omega^\prime$.
%
Then, by virtue of the first inequality in \eqref{p-p-1approx}, we have that
\begin{eqnarray}\label{bound}
	K_0\int_{\Omega^{\prime}}(|D{\uj}|^p-1)\,dx &\le& \int_{\Omega^{\prime}} \! \Bigl( F^j(x,D{\uj} ) + \bigl( |{\uj} |^2-a^2 \bigr)^{m}_{+} \Bigr)\,dx\cr\cr 
&\leq& \int_{\Omega^{\prime}} \!  F^j(x,Du )\,dx\le \int_{\Omega^{\prime}} \!  F(x,Du )\,dx,
\end{eqnarray}
where  last inequality is due to the monotonicity of the sequence $F^j$ given by \eqref{(II)}.
Hence the sequence $( D{\uj} )_j$ is bounded in $\LL^{p}(\Omega^{\prime}; \mathbb{R}^{nN})$ and there exists  $w\in \WW^{1,p}(\Omega^{\prime}; \mathbb{R}^N)$ such that $${\uj} \rightharpoonup w \qquad \text{weakly\,\,in}\quad \WW^{1,p}(\Omega^{\prime}; \mathbb{R}^N)\,\,\,\, \text{as}\,\,\,\,j\to +\infty.$$
Passing to the limit as $j\to+\infty$ in \eqref{conv1},  using the last inequality in \eqref{bound},  we also have that 
\begin{eqnarray}\label{conv2}
&&\limsup_{j\to+\infty}\int_{\Omega^{\prime}} \! \Bigl( F^j(x,D{\uj} ) + \bigl( |{\uj} |^2-a^2 \bigr)^{m}_{+} \Bigr)\,dx
\le \int_{\Omega^{\prime}} \!  F(x,Du )\,dx.
\end{eqnarray}
On the other hand, for every fixed $j_0\in \mathbb{N}$, the convexity of $\xi\to F^{j_0}(x,\xi)$, by lower semicontinuity, implies
\begin{eqnarray}\label{conv3}
\int_{\Omega^{\prime}} \!  F^{j_0}(x,Dw )\,dx &\le& \liminf_{j\to+\infty}\int_{\Omega^{\prime}} \!F^{j_0}(x,Du_j)\,dx\cr\cr
&\le& \liminf_{j\to+\infty}\int_{\Omega^{\prime}} \!F^{j}(x,Du_j)\,dx\cr\cr
&\le& \liminf_{j\to+\infty}\int_{\Omega^{\prime}} \! \Bigl( F^j(x,D{\uj} ) + \bigl( |{\uj} |^2-a^2 \bigr)^{m}_{+} \Bigr)\,dx\cr\cr
&\le&\int_{\Omega^{\prime}} \!  F(x,D{u} )\,dx,
\end{eqnarray}
where we used again the monotonocity of the sequence $F^j$ and  \eqref{conv2}. Taking the limit as $j_0\to\infty$ in the previous estimate, using the monotone convergence Theorem, we obtain
\begin{eqnarray}\label{conv4}
&&\int_{\Omega^{\prime}} \!  F(x,Dw )\,dx= \liminf_{j_0\to+\infty}\int_{\Omega^{\prime}} \!  F^{j_0}(x,Dw )\,dx\cr\cr
&\le& \int_{\Omega^{\prime}} \!  F(x,D{u} )\,dx\le \int_{\Omega^{\prime}} \!  F(x,D{w} )\,dx, 
\end{eqnarray}
by the minimality of $u$ and since $w=u$ on $\partial\Omega'$.
This, by  the strict convexity of $F$, yields that  $w\equiv u$ in $\Omega^\prime$. Hence, we conclude that
\begin{equation*}\label{conv}
u_j \rightharpoonup u \qquad \text{weakly\,\,in}\quad \WW^{1,p}(\Omega^{\prime}; \mathbb{R}^N).
\end{equation*} 
Using \eqref{conv4} in \eqref{conv3}, we have in particular that
\begin{eqnarray}\label{con1}
&&\lim_{j\to+\infty}\int_{\Omega^{\prime}} \! \bigl( |{\uj} |^2-a^2 \bigr)^{m}_{+} \,dx=0
\end{eqnarray}
which in turn implies
\begin{equation}\label{bounda}
  \sup_{j\in\mathbb{N}}\int_{\Omega^{\prime}} \!  |{\uj} |^{2m} \,dx  \le 2^m(1+|\Omega'|a^{2m})
\end{equation}
and also
\begin{eqnarray}\label{con2}
&&\lim_{j\to+\infty}\int_{\Omega^{\prime}} \!  F^j(x,D{u_j} )\,dx=\int_{\Omega^{\prime}} \!  F(x,D{u} )\,dx,
\end{eqnarray}
i.e. the first conclusion of the Lemma.
We also record that, by virtue of \eqref{(III)}, we have
\begin{eqnarray*}
&&\bar\nu\int_{\Omega'}(1+|Du|^2+|D\uj|^2)^{\frac{p-2}{2}}|Du-D\uj|^2\,dx\cr\cr
&\le& \int_{\Omega'}\Big( F^j(x,Du)-F^j(x,D\uj)+\langle D_\xi F^j(x,D\uj), D\uj-Du\rangle\Big)\,dx
\end{eqnarray*}
Since the Euler Lagrange system of the functional $\mathfrak{F}^j$ reads as
$$\int_{\Omega'}\langle D_\xi F^j(x,D\uj), D\varphi\rangle\,dx+2m\int_{\Omega'}(|\uj|^2-a^2)^{m-1}\uj\cdot\varphi\,dx =0$$ for all $\varphi = (\varphi^{\alpha})_{\alpha = 1, \dots, N} \in \CC_0^1(\Omega', \mathbb{R}^N)$, testing it with $\varphi=u-\uj,$ which is legitimate by density, we get
\begin{eqnarray*}
&&\bar\nu\int_{\Omega'}(1+|Du|^2+|Du_j|^2)^{\frac{p-2}{2}}|Du-Du_j|^2\,dx\cr\cr
&\le& \int_{\Omega'}\Big( F^j(x,Du)-F^j(x,Du_j)+\langle D_\xi F^j(x,Du_j), Du_j -Du\rangle\Big)\,dx\cr\cr
&=&\int_{\Omega'}\Big( F^j(x,Du)-F^j(x,D\uj)\Big)\,dx -2m\int_{\Omega'}(|\uj|^2-a^2)^{m-1}\uj(u-\uj)\,dx\cr\cr
&\le&\int_{\Omega'}\Big( F(x,Du)-F^j(x,D\uj)\Big)\,dx-2m\int_{\Omega'}(|\uj|^2-a^2)^{m-1}\uj(u-\uj)\,dx.
\end{eqnarray*}
Therefore, by \eqref{con1}, \eqref{bounda} and \eqref{con2}, taking the limit as $j\to +\infty$ in previous inequality, we conclude that
\begin{eqnarray*}
&&\limsup_{j\to +\infty}\int_{\Omega'}(1+|Du|^2+|D\uj|^2)^{\frac{p-2}{2}}|Du-D\uj|^2\,dx=0
\end{eqnarray*}
that is 
\begin{equation}\label{conv}
\uj \to u \qquad \text{strongly\,\,in}\quad \WW^{1,p}(\Omega^{\prime}; \mathbb{R}^N)	
\end{equation} 
which concludes the proof.
\end{proof}

The main tool in the proof of our main result  is the following Gagliardo--Nirenberg type inequality that we state
as a lemma and whose proof can be found in the Appendix A of \cite{CKP} (see also  \cite{GPdN}).
\begin{lemma}\label{gagnir}
For $\eta \in \CC^{1}_{c}(\Omega^{\prime})$ with $\eta \geq 0$ and $\CC^2$ maps
$v\colon \Omega^{\prime} \to \RN$ we have
\begin{eqnarray*}\label{intineq0}
\int_{\Omega^{\prime}} \! \eta^{2}|Dv|^{\frac{m}{m+1}(p+2)} \, dx
&\leq& (p+2)^{2}\left( \int_{\Omega^{\prime}}
\! \eta^{2}|v|^{2m} \, dx \right)^{\frac{1}{m+1}}\cr\cr
\times \left[ \left( \int_{\Omega^{\prime}}\! \eta^{2}
|D\eta |^{2}|Dv|^{p} \, dx \right)^{\frac{m}{m+1}} \right. &+& \left. n {N} \left(
\int_{\Omega^{\prime}}\! \eta^{2}|Dv|^{p-2}|D^{2}v|^{2} \, dx
\right)^{\frac{m}{m+1}}  \right]
\end{eqnarray*}
where $p\in (1, \infty)$ and $m>1$.
\end{lemma}
We conclude this subsection with a preliminary higher differentiability and a higher integrability result, that will be useful in the sequel.
\begin{theorem}\label{premhighdiff}
	Let ${\uj}\in \WW^{1,p}(\Omega';\mathbb{R}^N)\cap \LL^{2m}(\Omega';\mathbb{R}^N)$ be a local minimizer of $\mathfrak{F}^{j}(u,\Omega')$. Then
	$$(1+|D{\uj}|^2)^{\frac{p-2}{4}}|D{\uj}|\in \WW^{1,2}_{\mathrm{loc}}(\Omega') \quad \text{and}\qquad |D{\uj}|\in \LL^{\frac{m}{m+1}(p+2)}_{\mathrm{loc}}(\Omega').$$
\end{theorem}
For the proof we refer to  \cite{GPdN}.

\subsection{The Lipschitz continuity of the approximating minimizers}

\label{tre.tre}

Here, we establish the Lipschitz regularity of the approximating minimizers. Even tough such regularity is well known for minimizers of integral functionals satisfying standard growth conditions with respect to the gradient variable, the growth with respect to the $u$ variable in our energy density doesn't seem to fit with the available literature. The proof, however, relies on the very classical Moser iteration argument. More precisely, we have the following
\begin{theorem}\label{lipschitz}
Let ${\uj}\in \WW^{1,p}(\Omega';\R^N)\cap\LL^{2m}(\Omega';\R^N)$ be  a local minimizer of the functional \eqref{fun}. Then ${\uj}\in {\WW^{1,\infty}_{\loc}(\Omega'; \mathbb{R}^N)}$ with the estimate
$$||D{\uj}||_{\LL^\infty(B_R; \mathbb{R}^N)} \le M_{j}$$
for every ball $B_R\Subset\Omega'$ with a constant $M_j$ depending on $j$.
\end{theorem}
\begin{proof}
Testing the Euler--Lagrange system of the functional $\mathfrak{F}^j(v,\Omega)$ with the function
$\psi^\alpha = D_{x_s}\varphi^\alpha$ with $s\in\{1,\dots,n\}$, $\alpha \in \{1, \dots, N\}$ we get
\begin{eqnarray*}
0&=&\int_{\Omega'}\Big\langle \sum_{i,\alpha} F^j_{\xi^\alpha_i}(x, D{\uj}),D_{x_ix_s}\varphi^\alpha\Big\rangle\,dx\cr\cr 
&&+2m\int_{\Omega'}\sum_{\alpha}\bigl( |{\uj}|^2-a^2\bigr)^{m-1}_{+}{\uj}^\alpha  \cdot D_{x_s}\varphi^\alpha\,dx,\end{eqnarray*}
for every $\varphi\in C_0^1(\Omega';\mathbb{R}^N)$.
By Theorem \ref{premhighdiff}, we have that
$
{\uj} \in \WW^{2,2}_{\loc}(\Omega'; \mathbb{R}^N),
$
therefore integrating by parts the  integrals in previous identity, we get
\begin{eqnarray}
\label{second_variation}
&&\int_{\Omega'} \left (\sum_{i,\ell,\alpha,\beta}  F^j_{\xi^\alpha_i \xi^\beta_{\ell}}(x, D{\uj})D_{x_{\ell}x_s}({{\uj}^\beta})\varphi^\alpha_{x_i}   \, dx + \sum_{i,\alpha} F^j_{\xi^\alpha_i x_s}(x, D{\uj}) \varphi^\alpha_{x_i}  \right )\,dx \cr\cr
&&+2m\int_{\Omega'}\sum_{\alpha}D_{x_s}\left(\bigl( |{\uj}|^2-a^2\bigr)^{m-1}_{+}{\uj}^\alpha\right)  \varphi^\alpha\,dx=0,
\end{eqnarray}
holds for all $s = 1, \dots, n$ and for all $ \varphi\in C_0^1(\Omega'; \mathbb{R}^N)$. For $\eta \in \mathcal{C}^1_0(\Omega')$ and $\gamma \ge 0$, by density we can test \eqref{second_variation} with the function  $\varphi^\alpha=\eta^2
\left(\mathcal{D}_k{\uj}\right)^{\gamma} D_{x_s}({\uj}^\alpha)$,  where we used the notation
\begin{equation}\label{troncata}
\mathcal{D}_k{\uj} {:=} \left(1+\min\left\{|D{\uj}|^2,k^2\right\}\right)^{\frac12}
\end{equation}
One can easily check that
\begin{eqnarray*}
 \varphi^\alpha_{x_i}  &=&2 \eta \eta_{x_i} {\left(\mathcal{D}_k{\uj}\right)^{\gamma}} D_{x_s}({{\uj}^\alpha}) \\&&+  \eta^2 \gamma {\left(\mathcal{D}_k{\uj}\right)^{\gamma-2}\chi_{\{|Du_j|\le k\}}} |D{\uj}| D_{x_i}(|D{\uj}|) D_{x_s}({\uj^\alpha})\\ 
 &&+ \eta^2 {\left(\mathcal{D}_k{\uj}\right)^{\gamma}} D_{x_s x_i}({\uj}^\alpha). 
\end{eqnarray*}

Inserting in \eqref{second_variation} we get:
\begin{eqnarray*}
0 &=& 2 \int_{\Omega}\eta  \left(\mathcal{D}_k{\uj}\right)^{\gamma} \sum_{i,\ell,\alpha,\beta} F^j_{\xi_i^{\alpha} \xi_{\ell}^{\beta}}(x, D{\uj}) D_{x_{\ell} x_s}({\uj}^\beta)  \eta_{x_i}  D_{x_s}({\uj}^\alpha) \, dx\\
&& +  \int_{\Omega} \eta^2 \left(\mathcal{D}_k{\uj}\right)^{\gamma} \sum_{i,\ell,\alpha,\beta} F^j_{\xi_i^{\alpha} \xi_{\ell}^{\beta}}(x, D{\uj}) D_{x_{\ell} x_s}({\uj}^\beta)  D_{x_sx_i}({\uj}^\alpha) \, dx\\
&& + \gamma \int_{\Omega} \eta^2   \left(\mathcal{D}_k{\uj}\right)^{\gamma-2}\chi_{\{|Du_j|\le k\}}\sum_{i,\ell,\alpha,\beta} F^j_{\xi_i^{\alpha} \xi_{\ell}^{\beta}}(x, D{\uj}) D_{x_{\ell} x_s}({\uj}^\beta)\cr\cr
&&\qquad\cdot|D{\uj}| D_{x_i}(|D{\uj}|) D_{ x_s}({\uj}^\alpha) \, dx\\
&& + 2 \int_{\Omega} \eta\left(\mathcal{D}_k{\uj}\right)^{\gamma}\sum_{i,\alpha} F^j_{\xi_i x_s}(x, D{\uj})  \eta_{x_i}   D_{ x_s}({\uj}^\alpha) \, dx\\
&& +  \int_{\Omega} \eta^2 \left(\mathcal{D}_k{\uj}\right)^{\gamma} \sum_{i,\alpha} F^j_{\xi_i^{\alpha} x_s }(x, D{\uj})  D_{ x_sx_i}({\uj}^\alpha) \, dx\\
&& + \gamma \int_{\Omega} \eta^2   \left(\mathcal{D}_k{\uj}\right)^{\gamma-2}\chi_{\{|Du_j|\le k\}}\sum_{i,\alpha} F^j_{\xi_i^{\alpha} x_s}(x, D{\uj})\cr\cr
&&\qquad\cdot|D{\uj}| D_{x_i}(|D{\uj}|) D_{ x_s}({\uj}^\alpha) \, dx \cr\cr
&&+2m\int_{\Omega'}\eta^2  \left(\mathcal{D}_k{\uj}\right)^{\gamma}\sum_{\alpha}D_{x_s}\left(\bigl( |{\uj}|^2-a^2\bigr)^{m-1}_{+}{\uj}^\alpha\right)\cdot   D_{ x_s}({\uj}^\alpha)\,dx\cr\cr
&=:& I_1+I_2+I_3+I_4+I_5+I_6+I_7.
\end{eqnarray*}
Now we sum  all the terms in the previous  equation  with respect to $s$ from 1 to $n$, and we still denote for simplicity by $I_1-I_7$ the corresponding integrals.
\\
Previous equality yields
\begin{equation}\label{start}
	I_2+I_3+I_7\le |{I}_1|+|{I}_4|+|{I}_5|+|{I}_6|.
\end{equation}
\\
{Let us estimate the term $I_3.$ First of all, we have that
\[
F^j_{\xi_i^{\alpha}\xi_{\ell}^{\beta}}(x, \xi) = \left (\frac{{\tilde F^j_{tt}}(x, |\xi|)}{|\xi|^2} - \frac{{\tilde F^j_t}(x, |\xi|)}{|\xi|^3} \right ) \xi_i^{\alpha} \xi_{\ell}^{\beta} + \frac{{\tilde F^j_t}(x, |\xi|)}{|\xi|} \delta_{\xi_i^{\alpha} \xi_{\ell}^{\beta}},
\]
{where ${\tilde F^j}$ is given by \eqref{(I)} of Lemma \ref{apprcupguimas2parte}.}
Therefore
\begin{eqnarray*}
&&\sum_{i, \ell, \alpha, \beta, s} F^j_{\xi_i^{\alpha} \xi_{\ell}^{\beta}} (x, Du_j) D_{x_s}(u_j^{\alpha}) D_{x_{\ell} x_s}(u_j^{\beta}) D_{x_i}(|Du_j|) |Du_j|\\
&=& \left (\frac{{\tilde F^j_{tt}}(x, |Du_j|)}{|Du_j|^2} - \frac{{\tilde F^j_t}(x, |Du_j|)}{|Du_j|^3} \right ) \sum_{i, \ell, \alpha, \beta, s} D_{x_i}(u_j^{\alpha}) D_{x_{\ell}}(u_j^{\beta}) D_{x_s}(u_j^{\alpha}) D_{x_{\ell} x_s}(u_j^{\beta}) D_{x_i}(|Du_j|) |Du_j| \\
&& + \frac{{\tilde F^j_t}(x, |Du_j|)}{|Du_j|} \sum_{i, \alpha, s } D_{x_s} (u_j^{\alpha}) D_{x_s x_i}(u_j^{\alpha}) D_{x_i} (|Du_j|) |Du_j|
\\
&=& \left (\frac{{\tilde F^j_{tt}}(x, |Du_j|)}{|Du_j|^2} - \frac{{\tilde F^j_{t}}(x, |Du_j|)}{|Du_j|^3} \right ) \sum_{i,  \alpha, s} D_{x_i}(u_j^{\alpha}) D_{x_s}(u_j^{\alpha})  D_{x_i}(|Du_j|) D_{x_s}(|Du_j|) |Du_j|^2\\
&& +{\tilde F^j_{t}}(x, |Du_j|) |Du_j|  \sum_i [D_{x_i}(|Du_j|)]^2 \\
&=& \left ({\tilde F^j_{tt}}(x, |Du_j|) - \frac{{\tilde F^j_{t}}(x, |Du_j|)}{|Du_j|} \right ) \sum_{\alpha} \left [\sum_{i} D_{x_i}(u_j^{\alpha})  D_{x_i}(|Du_j|)\right ]^2 \\
&& + \frac{{\tilde F^j_{t}}(x, |Du_j|)}{|Du_j|} |Du_j|^2 |D(|Du_j|)|^2 
\end{eqnarray*}
where we used the fact that
$$D_{x_{s}}(|D\uj|)|D\uj|=\sum_{\ell, \beta} D_{x_{\ell} x_s}({\uj^{\beta}}) D_{x_{\ell}}({\uj^{\beta}}).$$
Now, by Cauchy-Schwarz' inequality, we have
\begin{eqnarray*}
\sum_{\alpha} \left [\sum_{i} D_{x_i}(u_j^{\alpha})  D_{x_i}(|Du_j|)\right ]^2 &\le & \, \sum_{i, \alpha} (D_{x_i}(u_j^{\alpha}))^2 \, \sum_i (D_{x_i}(|Du_j|))^2 \\
&\le &\, |Du_j|^2 \, |D(|Du_j|)|^2
\end{eqnarray*}
therefore, by using Kato's inequality
\begin{equation}
    \label{confD2u}
    |D(|Du_j|)|^2 \le \, |D^2 u_j|,
\end{equation}
we obtain that
\begin{eqnarray*}
I_3\!\! &=&\!\! \gamma\!\! \int_{\Omega} \eta^2 \left(\mathcal{D}_k{\uj}\right)^{\gamma-2}\chi_{\{|Du_j|\le k\}}\!\!\sum_{i,\ell,\alpha,\beta, s} F^j_{\xi_i^{\alpha} \xi_{\ell}^{\beta}}(x, D{\uj}) D_{x_{\ell} x_s}({\uj}^\beta) |D{\uj}| D_{x_i}(|D{\uj}|) D_{ x_s}({\uj}^\alpha) \, dx\\
&\ge&  \gamma\int_{\Omega} \eta^2   \left(\mathcal{D}_k{\uj}\right)^{\gamma-2}\chi_{\{|Du_j|\le k\}} {\tilde F^j_{tt}}(x, |Du_j|) \sum_{\alpha} \left [\sum_{i} D_{x_i}(u_j^{\alpha})  D_{x_i}(|Du_j|)\right ]^2 \ge 0.
\end{eqnarray*}
}

A simple calculation shows that
\begin{eqnarray*}
 I_7&=& 2m\int_{\Omega'}\eta^2  \left(\mathcal{D}_k{\uj}\right)^{\gamma}\sum_{\alpha,s}\bigl( |{\uj}|^2-a^2\bigr)^{m-1}_{+}\cdot   |D_{ x_s}({{\uj}^\alpha})|^2\,dx \cr\cr
 &+&4m(m-1)\int_{\Omega'}\eta^2  \left(\mathcal{D}_k{\uj}\right)^{\gamma}\sum_{\alpha,s}\bigl( |{\uj}|^2-a^2\bigr)^{m-2}_{+}|{\uj}^\alpha|^2\cdot   |D_{ x_s}({\uj}^\alpha)|^2\,dx\ge 0.
\end{eqnarray*}
Therefore, estimate \eqref{start} implies
\begin{eqnarray}\label{ristart}
	I_2\le |{I}_1|+|{I}_4|+|{I}_5|+|{I}_6|.
\end{eqnarray}
By Cauchy-Schwarz' inequality, Young's inequality and the right inequality in assumption  (F2), we have
\begin{eqnarray}\label{I1}
|{I}_1| &=& 2\left |\int_{\Omega}  \eta  \left(\mathcal{D}_k{\uj}\right)^{\gamma} {\sum_{i,\ell, \alpha, \beta, s}} {F^j_{\xi_i^{\alpha}\xi_{\ell}^{\beta}}}(x, D{\uj}) {D_{x_{\ell} x_s}}({{\uj}^{\beta}}) \eta_{x_i} D_{x_s}({{\uj}^{\alpha}}) \, dx\right |\cr\cr
&\le& 2\int_{\Omega}  \eta  \left(\mathcal{D}_k{\uj}\right)^{\gamma}\left \{ {\sum_{i,\ell, \alpha, \beta,s}} {F^j_{\xi_i^{\alpha} \xi_{\ell}^{\beta}}}(x, D{\uj}) \eta_{x_i} \eta_{x_j} {D_{x_s}({\uj}^{\alpha}) D_{x_s}({\uj}^{\beta})}\right \}^{1/2}  \cr\cr
&& \times \, \left \{ {\sum_{i,\ell, \alpha, \beta,s}} {F^j_{\xi_i^{\alpha} \xi_{\ell}^{\beta}}}(x, D{\uj}) {D_{x_s x_i}({\uj}^{\alpha}) \, D_{x_s x_{\ell}}({\uj}^{\beta})}\right \}^{1/2} \, dx \cr\cr
&\le &  C   \int_{\Omega}  \left(\mathcal{D}_k{\uj}\right)^{\gamma}  {\sum_{i,\ell, \alpha, \beta,s}} 
 {F^j_{\xi_i^{\alpha} \xi_{\ell}^{\beta}}}(x, D{\uj}) \eta_{x_i} \eta_{x_j} {D_{x_s}({\uj}^{\alpha}) D_{x_s}({\uj}^{\beta})} \, dx\cr\cr
&&  + \frac{1}{2} \int_{\Omega} \eta^2  \left(\mathcal{D}_k{\uj}\right)^{\gamma}  {\sum_{i,\ell, \alpha, \beta,s}} {F^j_{\xi_i^{\alpha} \xi_{\ell}^{\beta}}}(x, D{\uj}) {D_{x_s x_i}({\uj}^{\alpha}) \, D_{x_s x_{\ell}}({\uj}^{\beta}) } \, dx \cr\cr
&\le& C(j) \int_{\Omega} |D \eta|^2 \, \left(\mathcal{D}_k{\uj}\right)^{\gamma} (1 + |D{\uj}|^2)^{\frac{p}{2}}  \, dx \cr\cr
&+&\!\!\!\frac{1}{2} \int_{\Omega} \eta^2  \left(\mathcal{D}_k{\uj}\right)^{\gamma}\!\! \! {\sum_{i,\ell, \alpha, \beta,s}} {F^j_{\xi_i^{\alpha} \xi_{\ell}^{\beta}}}(x, D{\uj}) {D_{x_s x_i}({\uj}^{\alpha}) \, D_{x_s x_{\ell}}({\uj}^{\beta}) } \, dx,
\end{eqnarray}
{where  the last bound is due to the first inequality in \eqref{(VII)}}. 
We can estimate $I_4$ and $I_5$ by Cauchy-Schwartz' inequality together with the first inequality at \eqref{(V)} and Young's inequality, as follows
\begin{eqnarray}\label{I4}
|{I}_4| &\le & 2 \int_{\Omega} \eta  \left(\mathcal{D}_k{\uj}\right)^{\gamma} \sum_{i, \alpha, s} \left|F^j_{\xi_i^{\alpha} x_s}(x, D{\uj}) \eta_{x_i} {D_{x_s}({\uj}^{\alpha})}\right| \, dx\cr\cr
&{\le}&  C(j) \,  \int_{\Omega} \eta h(x) \, \left(\mathcal{D}_k{\uj}\right)^{\gamma} (1 + |D{\uj}|^2)^{\frac{p-1}{2}}  {\sum_{i,\alpha, s}} |\eta_{x_i} {D_{x_s}({\uj}^{\alpha})} | \, dx \cr\cr
&\le& C(j) \int_{\Omega} \eta|D\eta| h(x) \left(\mathcal{D}_k{\uj}\right)^{\gamma}(1 + |D\uj|^2)^{\frac{p}{2}}\,dx\cr\cr
&\le& C(j)\int_{\Omega} |D\eta|^2 \left(\mathcal{D}_k{\uj}\right)^{\gamma} (1 + |D{\uj}|^2)^{\frac{p}{2}}\,dx\\\nonumber
&&+C(j)\int_{\Omega} \eta^2 h^2(x)\left(\mathcal{D}_k{\uj}\right)^{\gamma} (1 + |D{\uj}|^2)^{\frac{p}{2}}\,dx.
\end{eqnarray}
Moreover
\begin{eqnarray}\label{I5}
|{I}_5| &=& \left| \int_{\Omega} \eta^2 \left(\mathcal{D}_k{\uj}\right)^{\gamma}{\sum_{i, \alpha, s} F^j_{\xi_i^{\alpha} x_s}}(x, D{\uj})  D_{x_s x_i}({{\uj}^{\alpha}}) \, dx \right| \cr\cr
&{\le}& C(j) \int_{\Omega} \eta^2\,h(x)\left(\mathcal{D}_k{\uj}\right)^{\gamma} (1 + |D{\uj}|^2)^{\frac{p-1}{2}}  \left|{\sum_{i, \alpha, s} D_{x_s x_i}({\uj}^{\alpha})} \right|\, dx \cr\cr
&\le&\, C(j) \, \int_{\Omega} \eta^2 h(x) \left(\mathcal{D}_k{\uj}\right)^{\gamma}(1 + |D{\uj}|^2)^{\frac{p-1}{2}}  |D^2 {\uj}| \, dx \cr\cr
&=& C(j)\int_{\Omega} \left[\eta^2 \left(\mathcal{D}_k{\uj}\right)^{\gamma}(1+|D{\uj}|^2)^{\frac{p-2}{2}} |D^2 {\uj}|^2 \right]^{\frac{1}{2}} \left[\eta^2 h^2(x) \left(\mathcal{D}_k{\uj}\right)^{\gamma}(1+|D{\uj}|^2)^{\frac{ p}{2}}\right]^{\frac{1}{2}} \, dx \cr\cr
&\le& \varepsilon \int_{\Omega} \eta^2 \left(\mathcal{D}_k{\uj}\right)^{\gamma}(1 + |D{\uj}|^2)^{\frac{p-2}{2}}  |D^2 {\uj}|^2 \, dx\\\nonumber&& + C_\varepsilon(j)  \int_{\Omega} \eta^2 h^2(x)\left(\mathcal{D}_k{\uj}\right)^{\gamma}(1 + |D{\uj}|^2)^{\frac{ p}{2}} \, dx,
\end{eqnarray}
where $\varepsilon>0$ will be chosen later.
Finally, similar arguments give
\begin{eqnarray}\label{I6}|{I}_6| &=&   \gamma \,\left| \int_{\Omega} \eta^2 \chi_{\{|Du_j|\le k\}} \left(\mathcal{D}_k{\uj}\right)^{\gamma-2}  |D{\uj}| {\sum_{i, \alpha, s}  F^j_{\xi_i^{\alpha} x_s}}(x, D{\uj})  D_{x_i}(|D{\uj}|) {D_{x_s}({\uj}^{\alpha})} \, dx\right|\cr\cr
&\le&  \, \gamma \, \int_{\Omega} \eta^2 \chi_{\{|Du_j|\le k\}} \left(\mathcal{D}_k{\uj}\right)^{\gamma-1}\cr\cr   &&\qquad\cdot\sum_{i, \alpha, s}\left| F^j_{\xi_i^{\alpha} x_s}(x, D{\uj})  D_{x_i}(|D{\uj}|) {D_{x_s}({\uj}^{\alpha})} \right|\, dx\cr\cr
&\le&  C (j)\, \gamma \, \int_{\Omega} \eta^2  \chi_{\{|Du_j|\le k\}} \left(\mathcal{D}_k{\uj}\right)^{p-2+\gamma}  \, h(x) \, {\sum_{i,\alpha, s}^n D_{x_i}(|D{\uj}|) D_{x_s}(u_j^{\alpha})} \, dx\cr\cr
&\le&  C (j)\, \gamma \, \int_{\Omega} \, \eta^2 \chi_{\{|Du_j|\le k\}} \left(\mathcal{D}_k{\uj}\right)^{p-1+\gamma}   |D^2 {\uj}| \, h(x)\, dx\cr\cr
&\le& \varepsilon \int_{\Omega} \eta^2 \left(\mathcal{D}_k{\uj}\right)^{\gamma} |D^2u_j|^2 (1 + |D{\uj}|^2)^{\frac{p-2}{2}}  \, dx\\\nonumber
&&+ C_\varepsilon(j) \gamma^2 \,  \, \int_{\Omega} \eta^2  h^2(x)\left(\mathcal{D}_k{\uj}\right)^{\gamma} {(1 + |D{\uj}|^2)^{\frac{p}{2}}} \, dx,
\end{eqnarray}
where  the constants $C$ and $C_{\varepsilon}(j)$ are independent of $\gamma$  {and where in the third inequality we used the Cauchy-Schwarz' inequality and \eqref{confD2u}}.
Plugging  \eqref{I1}, \eqref{I4}, \eqref{I5}, \eqref{I6} in \eqref{ristart} we obtain
\begin{eqnarray*}
&& \int_{\Omega} \eta^2 \left(\mathcal{D}_k{\uj}\right)^{\gamma}{\sum_{i,\ell, \alpha, \beta, s} F_{\xi_i^{\alpha} \xi_{\ell}^{\beta}}}(x, D{\uj}) {D_{x_{\ell} x_s}({\uj}^{\beta})  D_{x_s x_i}({\uj}^{\alpha})}\, dx\cr\cr
&\le&  \frac{1}{2} \int_{\Omega} \eta^2  \left(\mathcal{D}_k{\uj}\right)^{\gamma}  \, {\sum_{i,\ell, \alpha, \beta, s} F_{\xi_i^{\alpha} \xi_{\ell}^{\beta}}}(x, D{\uj}) {D_{x_{\ell} x_s}({\uj}^{\beta})  D_{x_s x_i}({\uj}^{\alpha})} \, dx \cr\cr
&&+2\varepsilon \int_{\Omega} \eta^2 \left(\mathcal{D}_k{\uj}\right)^{\gamma}(1 + |D{\uj}|^2)^{\frac{p-2}{2}} |D^2 {\uj}|^2 \, dx \cr\cr
&&+ C_\varepsilon(j) (1+\gamma^2)  \, \int_{\Omega} \eta^2  h^2(x)\left(\mathcal{D}_k{\uj}\right)^{\gamma}(1 + |D{\uj}|^2)^{\frac{p}{2} }\, dx\cr\cr
&&+C(j) \int_{\Omega} |D \eta|^2 \,  \left(\mathcal{D}_k{\uj}\right)^{\gamma}(1 + |D{\uj}|^2)^{\frac{p}{2}}  \, dx.
\end{eqnarray*}
Reabsorbing the first integral in the right hand side by the left hand side, we get
\begin{eqnarray*}
&& \frac{1}{2}\int_{\Omega} \eta^2\left(\mathcal{D}_k{\uj}\right)^{\gamma}{\sum_{i,\ell, \alpha, \beta, s} F_{\xi_i^{\alpha} \xi_{\ell}^{\beta}}}(x, D{\uj}) {D_{x_{\ell} x_s}({\uj}^{\beta})  D_{x_s x_i}({\uj}^{\alpha})}\, dx\cr\cr
 &\le &2\varepsilon \int_{\Omega} \eta^2 \left(\mathcal{D}_k{\uj}\right)^{\gamma}(1 + |D{\uj}|^2)^{\frac{p-2}{2}}  |D^2 {\uj}|^2 \, dx \cr\cr
&&+ C_\varepsilon(j) (1+\gamma^2) \,  \int_{\Omega} \eta^2  h^2(x)\left(\mathcal{D}_k{\uj}\right)^{\gamma}(1 + |D{\uj}|^2)^{\frac{
p}{2} }  \, dx\cr\cr
&&+C(j) \int_{\Omega} |D \eta|^2 \,  \left(\mathcal{D}_k{\uj}\right)^{\gamma}(1 + |D{\uj}|^2)^{\frac{p}{2}} \, dx.
\end{eqnarray*}
Using \eqref{(III)} in the left hand side of previous estimate, we obtain
\begin{eqnarray*}
&& \bar\nu\int_{\Omega} \eta^2\left(\mathcal{D}_k{\uj}\right)^{\gamma} (1 + |D{\uj}|^2)^{\frac{p-2}{2}}  |D^2 {\uj}|^2  \, dx\cr\cr
 &\le &2\varepsilon \int_{\Omega} \eta^2\left(\mathcal{D}_k{\uj}\right)^{\gamma} (1 + |D{\uj}|^2)^{\frac{p-2}{2}}  |D^2 {\uj}|^2 \, dx \cr\cr
&&+ C_\varepsilon(j) (1+\gamma^2) \, \int_{\Omega} \eta^2  h^2(x)\left(\mathcal{D}_k{\uj}\right)^{\gamma}(1 + |D{\uj}|^2)^{\frac{p}{2} }  \, dx\cr\cr
&&+C(j) \int_{\Omega} |D \eta|^2 \left(\mathcal{D}_k{\uj}\right)^{\gamma}  (1 + |D{\uj}|^2)^{\frac{p}{2}}  \, dx.
\end{eqnarray*}
Choosing $\varepsilon=\frac{ \bar\nu}{4}$, we can reabsorb the first integral in the right hand side by the left hand side thus getting
\begin{eqnarray}\label{Caccioppoli3a}
&& \int_{\Omega} \eta^2 \left(\mathcal{D}_k{\uj}\right)^{\gamma}(1 + |D{\uj}|^2)^{\frac{p-2}{2}}  |D^2 {\uj}|^2  \, dx\cr\cr
 &\le & C(j) (1+\gamma^2)  \, \int_{\Omega} \eta^2  h^2(x)\left(\mathcal{D}_k{\uj}\right)^{\gamma}(1 + |D{\uj}|^2)^{\frac{p}{2} }
 \, dx\cr\cr
&&+C(j) \int_{\Omega} |D \eta|^2 \,  \left(\mathcal{D}_k{\uj}\right)^{\gamma}(1 + |D{\uj}|^2)^{\frac{p}{2}}  \, dx,
\end{eqnarray}
with a constant $C$ dependent on $j$ but independent of $m$ and $\gamma$.

Let $0<r<R$, with $B_R\Subset\Omega'$ and fix $\eta\in C^1_0(B_{R})$ such that $\eta=1$ on $B_{r}$ and $|D\eta|\le \frac{C}{R-r}$ so that \eqref{Caccioppoli3a} implies
\begin{eqnarray*}\label{Caccioppoli3}
 && \int_{B_R} \eta^2 \left(\mathcal{D}_k{\uj}\right)^{\gamma}(1 + |D\uj|^2)^{\frac{p-2}{2}}  |D^2 \uj|^2 \, dx\cr\cr
&\le & C(j) (1+\gamma^2)  \, \int_{B_R}   h^2(x)\left(\mathcal{D}_k{\uj}\right)^{\gamma}(1 + |D\uj|^2)^{\frac{p}{2} }  \, dx\cr\cr
&&+\frac{C(j)}{(R-r)^2} \int_{B_R}  \,  \left(\mathcal{D}_k{\uj}\right)^{\gamma}(1 + |D\uj|^2)^{\frac{p}{2}}\, dx.
\end{eqnarray*}
The higher integrability result of Theorem \ref{premhighdiff}, recalling  the  assumption $h\in \LL^r_{\mathrm{loc}}(\Omega)$ and   choosing $\gamma$ such that  $(p+\gamma)\frac{r}{r-2}<\frac{m}{m+1}(p+2)$, i.e.  $\gamma<\frac{r-2}{r}\frac{m}{m+1}(p+2)-p$, allows us to pass to the limit as $k\to +\infty$ in both sides of previous estimate thus getting
\begin{eqnarray*}
 && \int_{B_R} \eta^2 (1 + |D\uj|^2)^{\frac{p-2+\gamma}{2}}  |D^2 \uj|^2 \, dx\cr\cr
&\le & C(j) (1+\gamma^2)  \, \int_{B_R}   h^2(x)(1 + |D\uj|^2)^{\frac{p+\gamma}{2} }  \, dx\cr\cr
&&+\frac{C(j)}{(R-r)^2} \int_{B_R}  \,  (1 + |D\uj|^2)^{\frac{p+\gamma}{2}}\, dx,
\end{eqnarray*}
since the sequence $\mathcal{D}_k{\uj}$ converges monotonically to $(1 + |D\uj|^2)^{\frac{1}{2}}$.
 Note that, by assumption \eqref{erre}, we may choose $m>\frac{rp}{2(r-p-2)}$ in order to have  $\gamma>0$.
\\
Using the Sobolev inequality in the left hand side of the previous estimate, we get
\begin{eqnarray}\label{Caccioppoli3}
 && \left ( \int_{B_{R}}  \eta^{2^*}(1 + |D\uj|^2)^{\left(\frac{p+\gamma}{4} \right) 2^*}\, dx\right )^{\frac{2}{2^*}} \cr\cr
&\le & C(j) (1+\gamma^2)  \, \int_{B_R} \eta^2  h^2(x)(1 + |D\uj|^2)^{\frac{p+\gamma}{2} }  \, dx\cr\cr
&&+\frac{C(j)}{(R-r)^2} \int_{\Omega'}   (1 + |D\uj|^2)^{\frac{p+\gamma}{2}}\, dx,
\end{eqnarray}
where we used the customary notation
\begin{equation}\label{sobexp}
 2^*=\begin{cases}
 \frac{2n}{n-2}\qquad\qquad\qquad\qquad\text{if}\,\,\, n>2\\
 \text{any finite exponent}\qquad \text{if}\,\,\, n=2.
\end{cases}   
\end{equation}
Since $h\in{\LL}^r_{\loc}(\Omega)$ with $r>n$, there exists $\vartheta\in (0,1)$ such that
$$\vartheta+\frac{2(1-\vartheta)}{2^*}+\frac{2}{r}=1\quad\Longleftrightarrow\quad \vartheta=1-\frac{2}{r}\frac{2^*}{2^*-2} $$ 
and therefore we use the interpolation inequality  to estimate the first integral in the right hand side of \eqref{Caccioppoli3}, as follows
\begin{eqnarray*}
  &&  \int_{B_R} \eta^2  h^2(x)(1 + |D\uj|^2)^{\frac{p+\gamma}{2} }  \, dx\cr\cr &=& \int_{B_R} \eta^2  h^2(x)(1 + |D\uj|^2)^{\left(\frac{p+\gamma}{2} \right)\vartheta}(1 + |D\uj|^2)^{\left(\frac{p+\gamma}{2} \right)(1-\vartheta)}  \, dx \cr\cr &\le& \left(\int_{B_R}   h^r\,dx\right)^{\frac2r}\left(\int_{B_R} \eta^2 (1 + |D\uj|^2)^{\frac{p+\gamma}{2}}\,dx\right)^\vartheta\cr\cr 
  &&\cdot\left(\int_{B_R} \eta^{2^*} (1 + |D\uj|^2)^{(\frac{p+\gamma}{4} ) 2^*}  \, dx\right)^{\frac{2(1-\vartheta)}{2^*}}.
\end{eqnarray*}
Inserting previous inequality in \eqref{Caccioppoli3}, by virtue of Young's inequality we obtain
\begin{eqnarray*}\label{Caccioppoli4}
 && \left ( \int_{B_{R}}  \eta^{2^*}(1 + |D\uj|^2)^{\left(\frac{p+\gamma}{4} \right) 2^*}\, dx\right )^{\frac{2}{2^*}} \cr\cr
&\le &\frac12\left(\int_{B_R} \eta^{2^*} (1 + |D\uj|^2)^{\left(\frac{p+\gamma}{4} \right) 2^*}  \, dx\right)^{\frac{2}{2^*}}\cr\cr
&&\quad+ {C(j, \vartheta)} (1+\gamma^2)^{{\frac{1}{\vartheta}}}   \left(\int_{B_R}   h^r\,dx\right)^{\frac{2}{r\vartheta}} \int_{B_R} \eta^2 (1 + |D\uj|^2)^{\frac{p+\gamma}{2}}\,dx \cr\cr 
&&+\frac{C(j)}{(R-r)^2} \int_{\Omega'}   (1 + |D\uj|^2)^{\frac{p+\gamma}{2}}\, dx.
\end{eqnarray*}
Reabsorbing the first integral in the right hand side by the left hand side, we get
\begin{eqnarray}\label{Caccioppoli4}
 && \left ( \int_{B_{R}}  \eta^{2^*}(1 + |D\uj|^2)^{\left(\frac{p+\gamma}{4} \right) 2^*}\, dx\right )^{\frac{2}{2^*}} \cr\cr
&\le &C(j) (1+\gamma^2)^{{\frac{1}{\vartheta}}}   \left(\int_{B_R}   h^r\,dx\right)^{{\frac{2}{r \vartheta}}}\int_{B_R} \eta^2 (1 + |D\uj|^2)^{\frac{p+\gamma}{2}}\,dx\cr\cr 
&&+\frac{C(j)}{(R-r)^2} \int_{\Omega'}   (1 + |D\uj|^2)^{\frac{p+\gamma}{2}}\, dx.
\end{eqnarray}
At this point it is quite standard to start the usual Moser    iteration  procedure to conclude with the desired Lipschitz continuity.
\end{proof}



\medskip
\medskip
\subsection{A Caccioppoli type inequality}

\label{tre.quattro}

This subsection is devoted to the proof of a second order  Caccioppoli type inequality for the approximating minimizers. It is worth mentioning that such inequality  is available in \cite{GPdN} and it is also the first part of the proof of Theorem \ref{lipschitz}, but the main point here is that we are going to establish it with  constants  independent  of the approximation parameters. More precisely, we have the following
\begin{lemma}\label{Caccio}
Let $\uj\in  \WW^{1,p}(\Omega'; \mathbb{R}^N)\cap \LL^{2m}(\Omega'; \mathbb{R}^N)$ be  a local minimizer of the functional $\mathfrak{F}^j(u,\Omega')$. Then, 
the following second order Caccioppoli type inequality
\begin{eqnarray}\label{Caccioppoli}
 && \int_{\Omega'} \eta^2 (1 + |D\uj|^2)^{\frac{p-2+\gamma}{2}} |D^2 {\uj}|^2 \, dx\cr\cr
&\le & C (1+\gamma^2)  \, \int_{\Omega'} \eta^2  h^2(x)(1 + |D{\uj}|^2)^{\frac{2q -p+\gamma}{2} } \, dx\cr\cr
&&+C \int_{\Omega'} |D \eta|^2 \,  (1 + |D{\uj}|^2)^{\frac{q+\gamma}{2}} \, dx,
\end{eqnarray}
holds true for every $\gamma \ge \, 0$ and for every $\eta\in C^1_0(\Omega)$, with a constant $C$ independent of $j$.
\end{lemma}
\begin{proof}
For $\eta \in \mathcal{C}^1_0(\Omega')$ and $\gamma \ge 0$, by the Lipschitz regularity of $u_j$ proven in Theorem \ref{lipschitz} and the higher differentiability result of Theorem \ref{premhighdiff},  we can test \eqref{second_variation} with the function  $\varphi^\alpha=\eta^2
\left(1+|D\uj|^2\right)^{\frac{\gamma}{2}} D_{x_s}({\uj}^\alpha)$. Arguing exactly as done in Theorem \ref{lipschitz} until inequality 
\eqref{ristart}, we arrive at
\begin{equation}\label{restart2}
 |\tilde I_2|\le |\tilde I_1|+|\tilde I_4|+|\tilde I_5|+|\tilde I_6|,    
\end{equation}
where
$$\tilde I_2=\int_{\Omega} \eta^2 \left(1+|D\uj|^2\right)^{\frac{\gamma}{2}}\sum_{i,\ell,\alpha,\beta} F^j_{\xi_i^{\alpha} \xi_{\ell}^{\beta}}(x, D{\uj}) D_{x_{\ell} x_s}({\uj}^\beta)  D_{x_sx_i}({\uj}^\alpha) \, dx$$
$$\tilde I_1=2\int_{\Omega}  \eta  (1 + |D{\uj}|^2)^{\frac{\gamma}{2}} {\sum_{i,\ell, \alpha, \beta, s}} {F^j_{\xi_i^{\alpha}\xi_{\ell}^{\beta}}}(x, D{\uj}) {D_{x_{\ell} x_s}}({{\uj}^{\beta}}) \eta_{x_i} D_{x_s}({{\uj}^{\alpha}}) \, dx$$
$$\tilde I_4=2 \int_{\Omega} \eta  (1 + |D{\uj}|^2)^{\frac{\gamma}{2}}  {\sum_{i, \alpha, s} F^j_{\xi_i^{\alpha} x_s}}(x, D{\uj}) \eta_{x_i} {D_{x_s}({\uj}^{\alpha})} \, dx$$
$$\tilde I_5= \int_{\Omega} \eta^2 (1 + |D{\uj}|^2)^{\frac{\gamma}{2}}{\sum_{i, \alpha, s} F^j_{\xi_i^{\alpha} x_s}}(x, D{\uj})  D_{x_s x_i}({{\uj}^{\alpha}}) \, dx  $$
and
$$\tilde I_6=\gamma \, \int_{\Omega}  (1+|D{\uj}|^2)^{\frac{\gamma}{2}-1} {\sum_{i, \alpha, s}F^j_{\xi_i^{\alpha} x_s}}(x, D{\uj}) \eta^2   |D{\uj}| D_{x_i}(|D{\uj}|) {D_{x_s}({\uj}^{\alpha})} \, dx.$$
\\
By the Cauchy-Schwartz inequality,  Young's inequality and the second inequality in \eqref{(VII)}, we have
\begin{eqnarray}\label{I1Caccioppoli}
|\tilde{I}_1| &=& 2\left |\int_{\Omega}  \eta  (1 + |D{\uj}|^2)^{\frac{\gamma}{2}} {\sum_{i,\ell, \alpha, \beta, s}} {F^j_{\xi_i^{\alpha}\xi_{\ell}^{\beta}}}(x, D{\uj}) {D_{x_{\ell} x_s}}({{\uj}^{\beta}}) \eta_{x_i} D_{x_s}({{\uj}^{\alpha}}) \, dx\right |\cr\cr
&\le& 2\int_{\Omega}  \eta  (1 + |D{\uj}|^2)^{\frac{\gamma}{2}}\left \{ {\sum_{i,\ell, \alpha, \beta,s}} {F^j_{\xi_i^{\alpha} \xi_{\ell}^{\beta}}}(x, D{\uj}) \eta_{x_i} \eta_{x_j} {D_{x_s}({\uj}^{\alpha}) D_{x_s}({\uj}^{\beta})}\right \}^{1/2}  \cr\cr
&& \times \, \left \{ {\sum_{i,\ell, \alpha, \beta,s}} {F^j_{\xi_i^{\alpha} \xi_{\ell}^{\beta}}}(x, D{\uj}) {D_{x_s x_i}({\uj}^{\alpha}) \, D_{x_s x_{\ell}}({\uj}^{\beta})}\right \}^{1/2} \, dx \cr\cr
&\le &  C   \int_{\Omega}  (1 + |D{\uj}|^2)^{\frac{\gamma}{2}} {\sum_{i,\ell, \alpha, \beta,s}} {F^j_{\xi_i^{\alpha} \xi_{\ell}^{\beta}}}(x, D{\uj}) \eta_{x_i} \eta_{x_j} {D_{x_s}({\uj}^{\alpha}) D_{x_s}({\uj}^{\beta})} \, dx\cr\cr
&&  + \frac{1}{2} \int_{\Omega} \eta^2  (1 + |D{\uj}|^2)^{\frac{\gamma}{2}}  {\sum_{i,\ell, \alpha, \beta,s}} {F^j_{\xi_i^{\alpha} \xi_{\ell}^{\beta}}}(x, D{\uj}) {D_{x_s x_i}({\uj}^{\alpha}) \, D_{x_s x_{\ell}}({\uj}^{\beta}) } \, dx \cr\cr
&\le& C(\tilde L) \int_{\Omega} |D \eta|^2 \,  (1 + |D{\uj}|^2)^{\frac{q+\gamma}{2}}  \, dx \cr\cr
& +&\!\!\! \frac{1}{2} \int_{\Omega} \eta^2  (1 + |D{\uj}|^2)^{\frac{\gamma}{2}} \!\! \! {\sum_{i,\ell, \alpha, \beta,s}} {F^j_{\xi_i^{\alpha} \xi_{\ell}^{\beta}}}(x, D{\uj}) {D_{x_s x_i}({\uj}^{\alpha}) \, D_{x_s x_{\ell}}({\uj}^{\beta}) } \, dx.
\end{eqnarray} 
We can estimate the fourth and the fifth term by Cauchy-Schwartz' inequality together with the second inequality in \eqref{(V)}, and  Young's inequality,  as follows
\begin{eqnarray}\label{I4Caccioppoli}
|\tilde{I}_4| &=& \left|2 \int_{\Omega} \eta  (1 + |D{\uj}|^2)^{\frac{\gamma}{2}}  {\sum_{i, \alpha, s} F^j_{\xi_i^{\alpha} x_s}}(x, D{\uj}) \eta_{x_i} {D_{x_s}({\uj}^{\alpha})} \, dx\right|\cr\cr
&{\le}&  C \,  \int_{\Omega} \eta h(x) \,  (1 + |D{\uj}|^2)^{\frac{q-1+\gamma}{2}}  {\sum_{i,\alpha, s}} |\eta_{x_i} {D_{x_s}({\uj}^{\alpha})} | \, dx \cr\cr
&\le&   C \,  \int_{\Omega} \eta |D\eta| |D{\uj}| \, h(x) \, (1 + |D{\uj}|^2)^{\frac{q-1+\gamma}{2}}  \, dx \cr\cr
&\le& C\int_{\Omega} |D\eta|^2  (1 + |D{\uj}|^2)^{\frac{p+\gamma}{2}}\,dx\\\nonumber
&&+C\int_{\Omega} \eta^2 h^2(x) (1 + |D{\uj}|^2)^{\frac{2q-p+\gamma}{2}}\,dx.
\end{eqnarray}
Moreover
\begin{eqnarray}\label{I5Caccioppoli}
|\tilde{I}_5| &=& \left| \int_{\Omega} \eta^2 (1 + |D{\uj}|^2)^{\frac{\gamma}{2}}{\sum_{i, \alpha, s} F^j_{\xi_i^{\alpha} x_s}}(x, D{\uj})  D_{x_s x_i}({{\uj}^{\alpha}}) \, dx \right| \cr\cr
&{\le}& C \int_{\Omega} \eta^2\,h(x) (1 + |D{\uj}|^2)^{\frac{q-1+\gamma}{2}}  \left|{\sum_{i, \alpha, s} D_{x_s x_i}({\uj}^{\alpha})} \right|\, dx \cr\cr
&\le&\, C \, \int_{\Omega} \eta^2 h(x) (1 + |D{\uj}|^2)^{\frac{q-1+\gamma}{2}}   |D^2 {\uj}| \, dx \cr\cr
&=& C\int_{\Omega} \left[\eta^2 (1+|D{\uj}|^2)^{\frac{p-2+\gamma}{2}} |D^2 {\uj}|^2 \right]^{\frac{1}{2}} \left[\eta^2 (1+|D{\uj}|^2)^{\frac{2q - p+\gamma}{2}} h^2(x)\right]^{\frac{1}{2}} \, dx \cr\cr
&\le& \varepsilon \int_{\Omega} \eta^2 (1 + |D{\uj}|^2)^{\frac{p-2+\gamma}{2}} |D^2 {\uj}|^2 \, dx\\\nonumber&& + C_\varepsilon  \int_{\Omega} \eta^2 h^2(x)(1 + |D{\uj}|^2)^{\frac{2q - p+\gamma}{2}} \, dx,
\end{eqnarray}
where $\varepsilon>0$ will be chosen later.
Finally, we have
\begin{eqnarray}\label{I6Caccioppoli}
|\tilde{I}_6| &=&   \gamma \,\left| \int_{\Omega} {\sum_{i, \alpha, s} F^j_{\xi_i^{\alpha} x_s}}(x, D{\uj}) \eta^2 (1+|D{\uj}|^2)^{\frac{\gamma}{2}-1}   |D{\uj}| D_{x_i}(|D{\uj}|) {D_{x_s}({\uj}^{\alpha})} \, dx\right|\cr\cr
&\le&  \, \gamma \, \int_{\Omega} \eta^2 (1+|D{\uj}|^2)^{\frac{\gamma-1}{2}}\cr\cr   &&\qquad\cdot{\sum_{i, \alpha, s}\left| F^j_{\xi_i^{\alpha} x_s}(x, D{\uj})  D_{x_i}(|D{\uj}|) {D_{x_s}({\uj}^{\alpha})} \right|}\, dx\cr\cr
&\le&  \, C\gamma \, \int_{\Omega} \eta^2 \, h(x) \, (1 + |D{\uj}|^2)^{\frac{q+\gamma-2}{2}}   \sum_{i,\alpha, s}^n \left|D_{x_i}(|D{\uj}|) D_{x_s}(u_j^{\alpha}) \right|\, dx\cr\cr
&\le& C \, \gamma \, \int_{\Omega} \, \eta^2 \, h(x)\,(1 + |D{\uj}|^2)^{\frac{q+\gamma}{2}}  |D^2 {\uj}|  dx\cr\cr
&\le& \varepsilon \int_{\Omega} \eta^2 |D^2u_j|^2 (1 + |D{\uj}|^2)^{\frac{p-2+\gamma}{2}}   \, dx\\\nonumber
&&+ C_\varepsilon \gamma^2 \,  \, \int_{\Omega} \eta^2  h^2(x) (1 + |D{\uj}|^2)^{\frac{2q -p+\gamma}{2}}\, dx,
\end{eqnarray}
where all the constants $C$ and $C_{\varepsilon}$ are independent of $\gamma$, { of $j$ and $m$} {and where in the third inequality we used Cauchy-Schwarz' inequality and \eqref{confD2u}}.
Plugging  \eqref{I1Caccioppoli}, \eqref{I4Caccioppoli}, \eqref{I5Caccioppoli}, \eqref{I6Caccioppoli} in \eqref{restart2} we obtain
\begin{eqnarray*}
&& \int_{\Omega} \eta^2 
(1 + |Du_j|^2)^{\frac{\gamma}{2}}{\sum_{i,\ell, \alpha, \beta, s} F_{\xi_i^{\alpha} \xi_{\ell}^{\beta}}}(x, D{\uj}) {D_{x_{\ell} x_s}({\uj}^{\beta})  D_{x_s x_i}({\uj}^{\alpha})}\, dx\cr\cr
&\le&  \frac{1}{2} \int_{\Omega} \eta^2  (1 + |Du_j|^2)^{\frac{\gamma}{2}} \, {\sum_{i,\ell, \alpha, \beta, s} F_{\xi_i^{\alpha} \xi_{\ell}^{\beta}}}(x, D{\uj}) {D_{x_{\ell} x_s}({\uj}^{\beta})  D_{x_s x_i}({\uj}^{\alpha})} \, dx \cr\cr
&&+2\varepsilon \int_{\Omega} \eta^2 (1 + |D{\uj}|^2)^{\frac{p-2+\gamma}{2}}  |D^2 {\uj}|^2 \, dx \cr\cr
&&+ C_\varepsilon (1+\gamma^2)  \, \int_{\Omega} \eta^2  h^2(x)(1 + |D{\uj}|^2)^{\frac{2q -p+\gamma}{2} }  \, dx\cr\cr
&&+C_\varepsilon \int_{\Omega} |D \eta|^2 \,  (1 + |D{\uj}|^2)^{\frac{q+\gamma}{2}}  \, dx.
\end{eqnarray*}
Reabsorbing the first integral in the right hand side by the left hand side we get
\begin{eqnarray*}
&& \frac{1}{2}\int_{\Omega} \eta^2(1 + |Du_j|^2)^{\frac{\gamma}{2}} {\sum_{i,\ell, \alpha, \beta, s} F_{\xi_i^{\alpha} \xi_{\ell}^{\beta}}}(x, D{\uj}) {D_{x_{\ell} x_s}({\uj}^{\beta})  D_{x_s x_i}({\uj}^{\alpha})}\, dx\cr\cr
 &\le &2\varepsilon \int_{\Omega} \eta^2 (1 + |D{\uj}|^2)^{\frac{p-2+\gamma}{2}} |D^2 {\uj}|^2 \, dx \cr\cr
&&+ C_\varepsilon (1+\gamma^2) \,  \int_{\Omega} \eta^2  h^2(x)(1 + |D{\uj}|^2)^{\frac{2q -p+\gamma}{2} }  \, dx\cr\cr
&&+C_\varepsilon \int_{\Omega} |D \eta|^2 \,  (1 + |D{\uj}|^2)^{\frac{q+\gamma}{2}}  \, dx.
\end{eqnarray*}
Using \eqref{(III)} in the left hand side of previous estimate, we obtain
\begin{eqnarray*}
&& \bar \nu\int_{\Omega} \eta^2 (1 + |D{\uj}|^2)^{\frac{p-2+\gamma}{2}}  |D^2 {\uj}|^2 \, dx \, dx\cr\cr
 &\le &2\varepsilon \int_{\Omega} \eta^2 (1 + |D{\uj}|^2)^{\frac{p-2+\gamma}{2}}  |D^2 {\uj}|^2 \, dx \cr\cr
&&+ C_\varepsilon (1+\gamma^2) \, \int_{\Omega} \eta^2  h^2(x)(1 + |D{\uj}|^2)^{\frac{2q -p+\gamma}{2} }  \, dx\cr\cr
&&+C_\varepsilon \int_{\Omega} |D \eta|^2 \,  (1 + |D{\uj}|^2)^{\frac{q+\gamma}{2}}  \, dx.
\end{eqnarray*}
Choosing $\varepsilon=\frac{\bar \nu}{4}$, we can reabsorb the first integral in the right hand side by the left hand side thus getting
\begin{eqnarray*}
&& \int_{\Omega} \eta^2 (1 + |D{\uj}|^2)^{\frac{p-2+\gamma}{2}}  |D^2 {\uj}|^2 \, dx \, dx\cr\cr
 &\le & C (1+\gamma^2)  \, \int_{\Omega} \eta^2  h^2(x)(1 + |D{\uj}|^2)^{\frac{2q -p+\gamma}{2} }  \, dx\cr\cr
&&+C \int_{\Omega} |D \eta|^2 \,  (1 + |D{\uj}|^2)^{\frac{q+\gamma}{2}}  \, dx,
\end{eqnarray*}
with a constant $C=C(\nu,\tilde L, n, N,p,q)$ {independent of $\gamma$, $j$ and $m$}.
This  concludes the proof.
\end{proof}

\section{The higher integrability}

\label{quattro}

Here, we establish an higher integrability result for the approximating minimizers with constants independent of the parameter of the approximation. This is the main point in achieving the proof of our main result. 
\begin{lemma}\label{lemhigh}
Let $\uj\in \LL^{2m}(\Omega'; \mathbb{R}^N)\cap \WW^{1,p}(\Omega'; \mathbb{R}^N)$ be  a local minimizer of the functional $\mathfrak{F}^{j}$ in \eqref{fun}. Setting 
\[
\mathfrak{m}_r := \frac{2 r m}{2m + r}.
\]
Then $${|D\uj|}\in \LL^{\mathfrak{m}_r(p-q+1)}_{\loc} (\Omega)
\qquad 
$$ with the following estimate
\begin{eqnarray}\label{lemhighest}
 &&\int_{B_\rho} |D\uj|^{\mathfrak{m}_r(p-q+1)}\, dx\le  \frac{C\Theta_R^{\frac{\mathfrak{m}_r}{2}}}{(R-\rho)^{r}}\left( \int_{B_R}
\! |\uj|^{2m} \, dx \right)^{\frac{\mathfrak{m}_r}{2m}}+C|B_R|,
 \end{eqnarray}
 for every balls $B_\rho\subset B_R\Subset \Omega'$ with a constant $C$ depending at most on $K_0,\overline{\nu},n,N,p,q,r$ but independent of $j$ and of $m$ and where we set
 \[
\Theta_R = \|1  + h\|^2_{\LL^r(B_R)}.
\]
\end{lemma}

\begin{proof}
By Theorem  \ref{lipschitz}, we have that ${\uj}\in  \WW^{1,\infty}_{\loc}(\Omega')$ and the Caccioppoli inequality at Lemma \ref{Caccio} yields that $$  (1 + |D{\uj}|^2)^{\frac{p-2}{2}+\gamma}  |D^2 {\uj}|^2\in \LL^1_{\loc}(\Omega'),$$
for every $\gamma>0$. Hence,   we are legitimate to apply Lemma \ref{gagnir} with $p$ replaced by $p+2\gamma$,  thus getting
\begin{eqnarray*}\label{intineq0}
&&\int_{\Omega^{\prime}} \! \eta^{2}|D\uj|^{\frac{m}{m+1}(p+2+2\gamma)} \, dx \cr\cr
&&\leq (p+2+2\gamma)^{2}\left( \int_{\Omega^{\prime}}
\! \eta^2|\uj|^{2m} \, dx \right)^{\frac{1}{m+1}} \left( \int_{\Omega^{\prime}}\! \eta^2
|D\eta |^{2}|D\uj|^{p+2\gamma} \, dx \right)^{\frac{m}{m+1}}\cr\cr
&&+ nN(p+2+2\gamma)^{2}\left( \int_{\Omega^{\prime}}
\! \eta^2|\uj|^{2m} \, dx \right)^{\frac{1}{m+1}} \!\left(
\int_{\Omega^{\prime}}\! \eta^2|D\uj|^{p-2+2\gamma}|D^{2}\uj|^{2} \, dx 
\right)^{\frac{m}{m+1}} ,
\end{eqnarray*}
for every non negative $\eta\in C^1_0(\Omega')$ such that $0\le \eta\le 1$. 
Using \eqref{Caccioppoli} to estimate the last  integral in the right hand side of the previous inequality, we obtain
\begin{eqnarray*}\label{intineq1}
&&\int_{\Omega^{\prime}} \! \eta^{2}|D\uj|^{\frac{m}{m+1}(p+2+2\gamma)} \, dx \cr\cr
&\le& (p+2+2\gamma)^{2}\left( \int_{\Omega^{\prime}}
\! \eta^2|\uj|^{2m} \, dx \right)^{\frac{1}{m+1}} \left( \int_{\Omega^{\prime}}\! \eta^2
|D\eta |^{2}|D\uj|^{p+2\gamma} \, dx \right)^{\frac{m}{m+1}}\cr\cr
&&C(p+2+2\gamma)^{4}\left( \int_{\Omega^{\prime}}
\! \eta^2|\uj|^{2m} \, dx \right)^{\frac{1}{m+1}}\left(\int_{\Omega'} \eta^2 h^2(x) (1 + |D\uj|)^{2q -p +2 \gamma}\,dx\right)^{\frac{m}{m+1}}\cr\cr  &&
+ C(p+2+2\gamma)^{2}\left( \int_{\Omega^{\prime}}
\! \eta^2|\uj|^{2m} \, dx \right)^{\frac{1}{m+1}} \left(
\int_{\Omega'}  |D\eta|^2 (1 + |D\uj|)^{{q + 2\gamma}}\,dx
\right)^{\frac{m}{m+1}},\cr\cr 
&\le& C(p+2+2\gamma)^{4}\left( \int_{\Omega^{\prime}}
\! \eta^2|\uj|^{2m} \, dx \right)^{\frac{1}{m+1}}\cr\cr
&&\qquad \cdot\left(\int_{\Omega'} \big(\eta^2 +|D\eta|^2\big)\big(1+h^2(x)\big) (1 + |D\uj|)^{2q -p +2 \gamma}\,dx\right)^{\frac{m}{m+1}}
\end{eqnarray*}
where we used that $1+\gamma\le p+2+2\gamma$ and that
$p+2\gamma\le q+2\gamma\le 2q-p+2\gamma$. 
\\
By virtue of the assumption on $h(x)$, we  use H\"older's inequality in the right hand side of the previous estimate thus getting
\begin{eqnarray}\label{intineq2}
&&\int_{\Omega^{\prime}} \! \eta^{2}|D\uj|^{\frac{m}{m+1}(p+2+2\gamma)} \, dx \cr\cr
&\le& C(p+2+2\gamma)^{4}\left( \int_{\Omega^{\prime}}
\! \eta^2|\uj|^{2m} \, dx \right)^{\frac{1}{m+1}}\left(\int_{\Omega'} \big(\eta^2 +|D\eta|^2\big)\big(1+h(x)\big)^r\,dx\right)^{\frac{2m}{r(m+1)}}\cr\cr
&&\qquad \cdot \left(\int_{\Omega'} \big(\eta^2 +|D\eta|^2\big)(1 + |D\uj|)^{\frac{r(2q -p +2 \gamma)}{r-2}}\,dx\right)^{\frac{m(r-2)}{r(m+1)}}
\end{eqnarray}
Fix concentric balls $B_\rho\subset B_s\subset B_t\subset B_R\Subset\Omega'$ and let $\eta\in C^1_0(B_t)$ be a standard cut off function between $B_s$ and $B_t$ i.e. $0\le \eta\le 1$, $\eta=1$ on $B_s$ and $|D\eta|\le \frac{c(n)}{t-s}$. Without loss of generality, we shall  suppose that $|B_R|\le 1$.
\\
With such a choice, estimate \eqref{intineq2} yields
\begin{eqnarray}\label{intineq2bis}
&&\int_{B_s} \! |D\uj|^{\frac{m}{m+1}(p+2+2\gamma)} \, dx \cr\cr
&\le& \frac{C\Theta_R^{\frac{m}{m+1}}}{(t-s)^{\frac{2m}{m+1}}}(p+2+2\gamma)^{4}\left( \int_{B_t}
\! |\uj|^{2m} \, dx \right)^{\frac{1}{m+1}}\cr\cr
&&\qquad \cdot \left(\int_{B_t} (1 + |D\uj|)^{\frac{r(2q -p +2 \gamma)}{r-2}}\,dx\right)^{\frac{m(r-2)}{r(m+1)}}
\end{eqnarray}
Note that we used the following
$$\eta^2 +|D\eta|^2\le 1+\frac{c}{(t-s)^2}\le \frac{c'}{(t-s)^2}, $$
since $t-s\le 1$.
Choose now $\gamma$ such that
 $$\frac{r(2q -p +2 \gamma)}{r-2}=\frac{m}{m+1}(p+2+2\gamma)\,\,\Longleftrightarrow\,\, 2\gamma=\frac{2mr(p-q+1)-2m(p+2)-r(2q-p)}{2m+r},$$
 which yields
 $$ \frac{m}{m+1}(p+2+2\gamma)=\frac{2rm}{2m+r}(p-q+1)=\mathfrak{m}_r(p-q+1)$$
Note that by virtue of \eqref{gap}, we have $\gamma>0$. Indeed
$$\gamma>0\,\,\Longleftrightarrow\,\,2mr(p-q+1)-2m(p+2)-r(2q-p)>0$$ 
$$\Longleftrightarrow\,\, 2m[r(p-q+1)-(p+2)]>(2q-p)r$$
The last  inequality can be satisfied for a suitable $m\in \mathbb{N}$ if
$$r(p-q+1)-(p+2)>0\,\,\Longleftrightarrow\,\,q<p+1-\frac{p+2}{r},$$
that holds true by virtue of assumption \eqref{gap}.
\\
With this choice of $\gamma$, observing that
\begin{equation}\label{boundm}
\frac{m+1}{r+2m}\le \frac12, \qquad\qquad \forall m\in\mathbb{N}    
\end{equation}
we have that
$$ p+2+2\gamma=\frac{2r(m+1)}{2m+r}(p-q+1)\le r(p-q+1),$$
and so estimate \eqref{intineq2bis} becomes
\begin{eqnarray}\label{intineq2ter}
&&\int_{B_s} \! |D\uj|^{\mathfrak{m}_r(p-q+1)} \, dx \cr\cr
&\le&\!\! \frac{C\Theta_R^{\frac{m}{m+1}}}{(t-s)^{\frac{2m}{m+1}}}\left( \int_{B_t}
\! |\uj|^{2m} \, dx \right)^{\frac{1}{m+1}}\!\!\left(\int_{B_t} (1 + |D\uj|)^{\mathfrak{m}_r(p-q+1)}\,dx\right)^{\frac{m(r-2)}{r(m+1)}}
\end{eqnarray}
with $C=C(K_0,\overline{\nu}, n,N,p,q,r)$  independent of $j$ and $m$. Using Young's inequality with exponents $$\left(\frac{r(m+1)}{m(r-2)};\frac{r(m+1)}{r+2m}\right)$$ in the right hand side
of the previous inequality, we obtain
\begin{eqnarray}\label{intineq2ter}
&&\int_{B_s} \! |D\uj|^{\mathfrak{m}_r(p-q+1)} \, dx \le \frac12 \int_{B_t} (1 + |D\uj|)^{\mathfrak{m}_r(p-q+1)}\,dx\cr\cr
&+& \frac{2^{\frac{m(r-2)}{r+2m}}C^{\frac{r(m+1)}{r+2m}}\Theta_R^{\frac{\mathfrak{m}_r}{2}}}{(t-s)^{\frac{2rm}{r+2m}}}\left( \int_{B_t}
\! |\uj|^{2m} \, dx \right)^{\frac{r}{r+2m}}
\cr\cr
&\le&\frac12 \int_{B_t} |D\uj|^{\mathfrak{m}_r(p-q+1)}\,dx+ |B_R|\cr\cr
&+&\!\! \frac{C\Theta_R^{\frac{\mathfrak{m}_r}{2}}}{(t-s)^{r}}\left( \int_{B_R}
\! |\uj|^{2m} \, dx \right)^{\frac{r}{r+2m}},
\end{eqnarray}
where, in order to control the constants, we used the bound at \eqref{boundm}, that $C\ge 1$ and that $R-\rho\le 1$.
Since previous estimate holds true for every $\rho<s<t<R$, we can apply Lemma \ref{lem:Giaq} thus obtaining
\begin{eqnarray*}
&&\int_{B_\rho} \! |D\uj|^{\mathfrak{m}_r(p-q+1)} \, dx \le \, \frac{C\Theta_R^{\frac{\mathfrak{m}_r}{2}}}{(R-\rho)^{r}}\left( \int_{B_R}
\! |\uj|^{2m} \, dx \right)^{\frac{r}{r+2m}}+c|B_R|,
\end{eqnarray*}
i.e. the conclusion.
\end{proof}

\begin{corollary}\label{lemhighsti}
Let ${\uj}\in \LL^{2m}_{\loc}(\Omega; \mathbb{R}^N)\cap \WW^{1,p}_{\loc}(\Omega; \mathbb{R}^N)$ be  a local minimizer of the functional $\mathfrak{F}^{j}$ in \eqref{fun}. Then $$|D{\uj}|\in {\LL}^{\mathfrak{m}_r(p-q+1)}_{\loc}(\Omega)$$ with the following estimate
\begin{eqnarray*}
 &&\int_{B_\rho} |D\uj|^{\mathfrak{m}_r(p-q+1)}\, dx\le  \frac{C\Theta_R^{\frac{\mathfrak{m}_r}{2}}}{(R-\rho)^{r}}\left( 1+a^{2m} \right)^{\frac{r}{r+2m}},
 \end{eqnarray*}
 for every balls $B_\rho\subset B_R\Subset \Omega'$, with $R\le 1$ and  with a constant $C$ depending at most on $K_0,p,q,r$ but independent of $j$ and of $m$.
\end{corollary}
\begin{proof}
It suffices to use \eqref{bounda} in the right hand side of 
estimate
 \eqref{lemhighest}.
\end{proof}

\section{Proof of Theorems \ref{main} and \ref{highdiff}}

\label{cinque}

We are now in position to establish the proof of our main result, that will be divided in two steps. In the first one we establish an uniform a priori estimate for the $\LL^\infty$ norm of the gradient of the minimizers of the approximating functionals while in the second we show that these estimates are preserved in passing to the limit.

\medskip
\begin{proof}[Proof of Theorem \ref{main}]

Let us fix a ball $B_{R_0}\Subset \Omega$ and radii $\frac{R_{0}}{2}<\bar\rho<\rho<t_1<t_2<R<\bar R<R_{0}\le 1$ that will be needed in the three iteration procedures, constituting the essential steps in our proof.

\noindent {\bf Step 1. The uniform a priori estimate.}
 Let us choose $\eta\in C^1_0(B_{t_2})$ such that $\eta=1$ on $B_{t_1}$ and $|D\eta|\le \frac{C}{t_2-t_1}$, so that \eqref{Caccioppoli} implies

\begin{eqnarray*}
&& \int_{B_{t_2}} \eta^2 (1 + |D{\uj}|^2)^{\frac{p-2}{2}+\gamma}  |D^2 {\uj}|^2 \, dx\cr\cr
&\le & C \frac{(1+\gamma^2)}{(t_2-t_1)^2}  \, \int_{B_{t_2}}  (1+h^2(x))(1 + |D{\uj}|^2)^{\frac{2q -p}{2}+\gamma }  \, dx.
\end{eqnarray*}
Using the assumptions on $h(x)$ and H\"older's inequality, we arrive at
\begin{eqnarray}\label{pinca}
&&\int_{B_{t_2}} \eta^2 (1 + |D{\uj}|^2)^{\frac{p-2}{2} + \gamma}  |D^2 {\uj}|^2 \, dx \cr\cr
&\le&   (1 + \gamma^2)\frac{C\Theta}{(t_2 - t_1)^{2}} \left [\int_{B_{t_2}}(1 + |D{\uj}|^2)^{\frac{(2\gamma + 2q - p)\mathfrak{r}}{2}}  \, dx \right
]^{\frac{1}{\mathfrak{r}}}
\end{eqnarray}
for any  $0 < t_1 < t_2$, where the constant $C$  is independent of $\gamma,$ of $m$ and of $\varepsilon$, where $\Theta=\Theta_{R_0}$, and where we set
$$\mathfrak{r} =\frac{r}{r-2}.$$
The Sobolev  inequality yields
\begin{eqnarray*}
&&\left ( \int_{B_{t_2}}  \eta^{2^*}(1 + |D{\uj}|^2)^{(\frac{p}{4} + \frac{\gamma}{2}) 2^*}\, dx\right )^{\frac{2}{2^*}} \le \, C \, \int_{B_{t_2}} |D (\eta(1 + |D{\uj}|^2)^{\frac{p}{4} + \frac{\gamma}{2}})|^2 \, dx 	\cr\cr
&\le&C(1+\gamma^2)\int_{B_{t_2}} \eta^2 (1 + |D{\uj}|^2)^{\frac{p-2}{2} + \gamma}  |D^2 {\uj}|^2\,dx+C \int_{B_{t_2}} |D \eta|^2 \,  (1 + |D{\uj}|^2)^{\frac{p}{2}+\gamma}  \, dx,
\end{eqnarray*}
where $2^*$ is the exponent defined at \eqref{sobexp}.
	Using estimate \eqref{pinca} to control the first integral in the right hand side of the previous inequality, we obtain
\begin{eqnarray}\label{stimapreite}
&&\left ( \int_{B_{t_2}}  \eta^{2^*}(1 + |D{\uj}|^2)^{(\frac{p}{4} + \frac{\gamma}{2}) 2^*}\, dx\right )^{\frac{2}{2^*}} 	\cr\cr
&\le& C\frac{\Theta(1+\gamma^4)}{(t_2 - t_1)^{2}}\left [\int_{B_{t_2}}(1 + |D{\uj}|^2)^{\frac{(2\gamma + 2q - p)\mathfrak{r}}{2}}  \, dx \right ]^{\frac{1}{\mathfrak{r}}}\cr\cr
&&+\frac{C}{(t_2 - t_1)^{2}} \int_{B_{t_2}}  \,  (1 + |D{\uj}|^2)^{\frac{p}{2}+\gamma}  \, dx\cr\cr
&\le& C\frac{\Theta(1+\gamma^4)}{(t_2 - t_1)^{2}}\left [\int_{B_{t_2}}(1 + |D{\uj}|^2)^{\frac{(2\gamma + 2q - p)\mathfrak{r}}{2}}  \, dx \right ]^{\frac{1}{\mathfrak{r}}},
\end{eqnarray}	
where we used that $p\le 2q-p$ and that $L^{\mathfrak{r}}\hookrightarrow L^1$.
Now, setting $$V(D{\uj})=(1+|D{\uj}|^2)^{\frac{1}{2}},$$ we can write \eqref{stimapreite} as follows
\begin{eqnarray*}
&&\left ( \int_{B_{R}}  \eta^{2^*}V(D{\uj})^{(p+2\gamma)\frac{2^*}{2}}\, dx\right )^{\frac{2}{2^*}} 	\cr\cr
&\le& C\frac{\Theta (1+\gamma^4)}{(R - \rho)^{2}}\left [\int_{B_R}V(D{\uj})^{[ 2(q - p)\mathfrak{r}+(p+2\gamma)\mathfrak{r}]}  \, dx \right ]^{\frac{1}{\mathfrak{r}}} \cr\cr
&\le& C\frac{\Theta (1+\gamma^4)}{(R - \rho)^{2}}||V(D{\uj})||^{2(q-p)}_{L^\infty(B_R)}\left [\int_{B_R}V(D{\uj})^{(p+2\gamma)\mathfrak{r}}  \, dx \right ]^{\frac{1}{\mathfrak{r}}}
\end{eqnarray*}
and so
\begin{eqnarray}\label{stimapreiteter}
&&\left ( \int_{B_{\rho}}  V(D{\uj})^{[\mathfrak{r}(p+2\gamma)]\frac{2^*}{2\mathfrak{r}}}\, dx\right )^{\frac{2\mathfrak{r}}{2^*}} 	\cr\cr
&\le& C {\left [\frac{\Theta (1+\gamma^4)}{(R - \rho)^{2}} \right ]^\mathfrak{r}} ||V(D{\uj})||^{2\mathfrak{r}(q-p)}_{L^\infty(B_R)}\int_{B_R}V(D{\uj})^{(p+2\gamma)\mathfrak{r}}  \, dx ,
\end{eqnarray}	
where we also used that $\eta=1$ on $B_\rho$.
\\
{With $\frac{R_0}{2}\le \bar \rho<\bar R\le R_0$ fixed at the beginning of the section, we} define the decreasing sequence of radii by setting
$$\rho_i=\bar \rho+\frac{\bar R-\bar \rho}{2^i}.$$
Let us also define the following increasing sequence of exponents
$$p_0=p\mathfrak{r}\qquad\qquad {p_{i}}={p_{i-1}}\frac{2^*}{2\mathfrak{r}}=p_0\left(\frac{2^*}{2\mathfrak{r}}\right)^{i}$$
Noticing that, since $u_j \in \WW^{1,\infty}_{\mathrm{loc}}(\Omega)$, estimate \eqref{stimapreiteter} holds true for $\gamma=0$ and for every $\bar \rho<\rho<R<\bar R$, we may iterate it  on the concentric balls $B_{\rho_i}$  with exponents $p_i$,
thus obtaining
\begin{eqnarray}\label{stimapreitepassoi}
&&\left ( \int_{B_{\rho_{i+1}}}  V(D{\uj})^{p_{i+1}}\, dx\right )^{\frac{1}{p_{i+1}}} 	\cr\cr
&\le& \displaystyle{\prod_{h=0}^{i}}\left(C {\frac{\Theta^\mathfrak{r} p_h^{4\mathfrak{r}}}{(\rho_h - \rho_{h+1})^{2\mathfrak{r}}}}||V(D{\uj})||^{2\mathfrak{r}(q-p)}_{L^\infty(B_R)}\right)^{\frac{1}{p_h}}\left(\int_{B_{\rho_0}}V(D{\uj})^{p_0}  \,
dx\right)^{\frac{1}{p_0}}\cr\cr
&=& \displaystyle{\prod_{h=0}^{i}}\left(C{\frac{4^{h+1}\Theta^\mathfrak{r} p_h^{4\mathfrak{r}}}{(\bar R-\bar \rho )^{2\mathfrak{r}}}}||V(D{\uj})||^{2\mathfrak{r}(q-p)}_{L^\infty(B_R)}\right)^{\frac{1}{p_h}}\left(\int_{B_{\rho_0}}V(D{\uj})^{p_0}  \,
dx\right)^{\frac{1}{p_0}}\cr\cr
&=&\displaystyle{\prod_{h=0}^{i}}\!\!\left(4^{h+1} {p_h^{4\mathfrak{r}}} \right)^{\frac{1}{p_h}}\displaystyle{\prod_{h=0}^{i}}\left({\frac{C\Theta^\mathfrak{r} }{(\bar R-\bar \rho
)^{2\mathfrak{r}}}}||V(D{\uj})||^{2\mathfrak{r}(q-p)}_{L^\infty(B_R)}\right)^{\frac{1}{p_h}}\cr\cr
&&\qquad\qquad\cdot \left(\int_{B_{\rho_0}}V(D{\uj})^{p_0} dx\right)^{\frac{1}{p_0}}
\end{eqnarray}
Since
$$\displaystyle{\prod_{h=0}^{i}}\left(4^{h+1} {p_h^{4\mathfrak{r}}}\right)^{\frac{1}{p_h}}=\exp\left(\sum_{h=0}^i \frac{1}{p_h}\log(4^{h+1} {p_h^{4\mathfrak{r}}})\right)\le \exp\left(\sum_{h=0}^{+\infty} \frac{1}{p_h}\log({4}^{h+1}
{p_h^{4\mathfrak{r}}})\right) \le c(r)$$
and
\begin{eqnarray*}
	&&\displaystyle{\prod_{h=0}^{i}}\left(\frac{C {\Theta^\mathfrak{r}} }{{(\bar R-\bar \rho )^{2\mathfrak{r}}}}||V(D{\uj})||^{2\mathfrak{r}(q-p)}_{L^\infty(B_R)}\right)^{\frac{1}{p_h}}=\left(\frac{C {\Theta^\mathfrak{r}} }{{(\bar R-\bar \rho
)^{2\mathfrak{r}}}}||V(Du_{j})||^{2\mathfrak{r}(q-p)}_{L^\infty(B_R)}\right)^{\sum_{h=0}^{i}\frac{1}{p_h}}
	\cr\cr
	&\le& \left(\frac{C {\Theta^\mathfrak{r}} }{{(\bar R-\bar \rho )^{2\mathfrak{r}}}}||V(D{\uj})||^{2\mathfrak{r}(q-p)}_{L^\infty(B_R)}\right)^{\sum_{h=0}^{+\infty}\frac{1}{p_h}}=\left(\frac{C {\Theta^\mathfrak{r}} }{{(\bar R-\bar \rho
)^{2\mathfrak{r}}}}||V(D{\uj})||^{2\mathfrak{r}(q-p)}_{L^\infty(B_R)}\right)^{\frac{2^*}{p_0(2^* - 2 \mathfrak{r})}},
\end{eqnarray*}
we can let $i\to \infty$ in \eqref{stimapreitepassoi} thus getting
\begin{eqnarray*}
	||V(D{\uj})||_{L^{\infty}(B_{\bar\rho})} \le C \left(\frac{\Theta }{(\bar R-\bar \rho )^{2}}\right)^{\frac{2^*\mathfrak{r}}{p_0(2^*-2\mathfrak{r})}}||V(D{\uj})||_{L^{\infty}(B_{\bar R})}^{\frac{2\cdot 2^*\mathfrak{r}(q-p)}{p_0(2^*-2m)}}\left(\int_{B_{\bar R}}V(D{\uj})^{p_0}
\, dx\right)^{\frac{1}{p_0}},
\end{eqnarray*}
since $\sum_{h=0}^{{\infty}}\frac{1}{p_h}= \frac{2^*}{p_0(2^*-2\mathfrak{r})} $. Therefore, {using the fact that  $p_0=p\mathfrak{r}$, we deduce}
\begin{eqnarray*}
	||V(D{\uj})||_{L^{\infty}(B_{\bar\rho})} \le C \left(\frac{\Theta }{(\bar R-\bar \rho )^{2}}\right)^{{\frac{2^*}{p(2^*-2\mathfrak{r})}}}||V(D{\uj})||_{L^{\infty}(B_{\bar R})}^{\frac{2\cdot 2^*(q-p)}{p(2^*-2\mathfrak{r})}}\left(\int_{B_{\bar R}}V(D{\uj})^{p\mathfrak{r}}
\, dx\right)^{\frac{1}{p\mathfrak{r}}}.
\end{eqnarray*}  
Using the higher integrability at Lemma \ref{lemhigh} we deduce that
\begin{eqnarray}\label{stimaquasifin}
	||V(D{\uj})||_{L^{\infty}(B_{\bar\rho})} &\le& C \left(\frac{\Theta }{(\bar R-\bar \rho )^{2}}\right)^{{\frac{2^*}{p(2^*-2\mathfrak{r})}}}||V(D{\uj})||_{L^{\infty}(B_{\bar R})}^{\frac{2\cdot 2^*(q-p)}{p(2^*-2\mathfrak{r})}-{\frac{2m}{2m+r}\frac{r(p-q+1)}{p\mathfrak{r}}}+1}\cr\cr
	&&\qquad\cdot\left(\int_{B_{\bar R}}V(D{\uj})^{{\frac{2mr}{2m+r}(p-q+1)}}
\, dx\right)^{\frac{1}{p\mathfrak{r}}}.
\end{eqnarray}  
Now,  we note that
{$$\frac{2\cdot 2^*(q-p)}{p(2^*-2\mathfrak{r})}+1-\frac{r(p-q+1)}{p\mathfrak{r}}<1\,\,\Longleftrightarrow\,\, \frac{2\cdot 2^*(q-p)}{2^*-2\mathfrak{r}}<(r-2)(p-q+1)$$
\begin{equation}\label{disgoal}
 \Longleftrightarrow\,\,(q-p)\left[\frac{2\cdot 2^*}{2^*-2\mathfrak{r}}+r-2\right]< r-2\,\,\Longleftrightarrow\,\,q-p< \frac{r-2}{\frac{2\cdot 2^*}{2^*-2\mathfrak{r}}+r-2}   
\end{equation}
where we used that $\mathfrak{r}=\frac{r}{r-2}$}.
Since
$$\frac{2\cdot 2^*}{2^*-2\mathfrak{r}}=\frac{2\cdot\frac{2n}{n-2}}{\frac{2n}{n-2}-\frac{2r}{r-2}}=\frac{\frac{2n}{n-2}}{\frac{n}{n-2}-\frac{r}{r-2}}=\frac{\frac{2n}{n-2}}{\frac{2(r-n)}{(n-2)(r-2)}}=\frac{n(r-2)}{r-n}, $$
last inequality is equivalent to
$$q-p< \frac{r-2}{\frac{n(r-2)}{r-n}+r-2} =\frac{1}{\frac{n}{r-n}+1}=\frac{r-n}{r}=1-\frac{n}{r}$$
that is of course satisfied under our assumption on the gap at \eqref{gap}. 
Note that in case $n=2$,  the bound \eqref{gap} reads as   
$$q<p+1-\max\left\{\frac{2}{r},\frac{p+2}{r}\right\},$$
and one can easily check that inequality \eqref{disgoal} is fulfilled provided that we choose $2^*<2$.
By the inequality
$$\frac{2\cdot 2^*(q-p)}{p(2^*-2\mathfrak{r})}<\frac{r(p-q+1)}{p\mathfrak{r}},$$
we can determine $m$ large enough so that
$$\frac{2\cdot 2^*(q-p)}{p(2^*-2\mathfrak{r})}<\frac{2m}{2m+r}\frac{r(p-q+1)}{p\mathfrak{r}}.$$
Indeed it suffices to choose
\begin{equation}\label{mlarge}
 m>\frac{2^*r\mathfrak{r}(q-p)}{r(2^*-2\mathfrak{r})(p-q+1)-2^*2\mathfrak{r}(q-p)}   
\end{equation}
For $m$ satisfying the previous inequality, to simplify the notation we set
$$\chi_m =\frac{2m}{2m+r}\frac{r(p-q+1)}{p\mathfrak{r}}-\frac{2\cdot 2^*(q-p)}{p(2^*-2\mathfrak{r})} $$
so that estimate \eqref{stimaquasifin} can be expressed as
\begin{eqnarray}\label{stimaquasifin2}
	||V(D{\uj})||_{L^{\infty}(B_{\bar\rho})} &\le& C \left(\frac{\Theta }{(\bar R-\bar \rho )^{2}}\right)^{{\frac{2^*}{p(2^*-2\mathfrak{r})}}}||V(D{\uj})||_{L^{\infty}(B_{\bar R})}^{1-\chi_m}\cr\cr
	&&\qquad\cdot\left(\int_{B_{\bar R}}V(D{\uj})^{{\frac{2mr}{2m+r}(p-q+1)}}
\, dx\right)^{\frac{1}{p\mathfrak{r}}},
\end{eqnarray}
with $1-\chi_m\in (0,1).$
Hence, we can use  Young's inequality  with exponents $$\frac{1}{1-\chi_m}\quad \text{and}\quad \frac{1}{\chi_m}$$ in the right hand side of \eqref{stimaquasifin}, we get
 \begin{eqnarray}\label{stimaquasifin00}
&&	||V(D{\uj})||_{L^{\infty}(B_{\bar\rho})} \le \frac{1}{2} ||V(D{\uj})||_{L^{\infty}(B_{\bar R})}\\
&&+\left(\frac{C\Theta }{(\bar R-\bar \rho )^{2}}\right)^{{\frac{2^*}{\chi_{_m} p(2^*-2\mathfrak{r})}}}\left(\int_{B_{\bar R}}V(D{\uj})^{{\frac{2mr}{2m+r}(p-q+1)}}  \, dx\right)^{\frac{1}{p\mathfrak{r}\chi_{_m}}}. \nonumber
\end{eqnarray}
 Since the previous estimate holds true for every $\frac{R_0}{2}<\bar\rho<\bar R<R_0$, by Lemma \ref{lem:Giaq} we get
\[
||V(D{\uj})||_{L^{\infty}\left(B_{\frac{R_0}{2}}\right)} \le \left(\frac{C\Theta }{R_0^{2}}\right)^{{\frac{2^*}{\chi_{_m} p(2^*-2\mathfrak{r})}}}\left(\int_{B_{ R_0}} V(D{\uj})^{{\frac{2mr}{2m+r}(p-q+1)}}  \, dx\right)^{\frac{1}{p\mathfrak{r}\chi_{_m}}}
\]
and by Corollary \ref{lemhighsti}
\[
||V(D{\uj})||_{L^{\infty}\left(B_{\frac{R_0}{2}}\right)} \le C\left(\frac{\Theta }{R_0^{2}}\right)^{{\frac{2^*}{\chi_{_m} p(2^*-2\mathfrak{r})}}}\Bigg(\frac{\Theta^{\frac{\mathfrak{m}_r}{2}}}{R_0^{r}}\left( 1+a^{2m} \right)^{\frac{r}{r+2m}}\Bigg)^{\frac{1}{p\mathfrak{r}\chi_{_m}}}
\]
with a constant $C$ independent of $m$ and $j$.

\medskip
\noindent {\bf Step 2. The passage to the limit.}
Recalling \eqref{conv}, taking the limit as $j\to \infty$ in the previous estimate , we have
 \begin{eqnarray*}\label{stimaquasifin4}
	&&||V(Du)||_{L^{\infty}\left(B_{\frac{R_0}{2}}\right)}\le \liminf_{j\to \infty}||V(D{\uj})||_{L^{\infty}\left(B_{\frac{R_0}{2}}\right)}\cr\cr 
 &\le& C\left(\frac{\Theta }{R_0^{2}}\right)^{{\frac{2^*}{\chi_{_m} p(2^*-2\mathfrak{r})}}}\Bigg(\frac{\Theta^{\frac{\mathfrak{m}_r}{2}}}{R_0^{r}}\left( 1+a^{2m} \right)^{\frac{r}{r+2m}}\Bigg)^{\frac{1}{p\mathfrak{r}\chi_{_m}}}.
\end{eqnarray*}
Now, we observe
$$\lim_{m\to \infty}\chi_m=\frac{r(p-q+1)}{p\mathfrak{r}}-\frac{2\cdot 2^*(q-p)}{p(2^*-2\mathfrak{r})}=:\chi_1(p,q,r,n)$$
$$\lim_{m\to \infty}\frac{2mr(p-q+1)}{(2m+r)}=r(p-q+1)$$
Since previous estimate holds for every $m$ large enough to satisfy \eqref{mlarge} with constant independent of $m$, we now take the limit as $m\to\infty$ thus getting
 \begin{eqnarray}\label{stimafin4}
	&&||V(Du)||_{L^{\infty}\left(B_{\frac{R_0}{2}}\right)}
 \le C(\Theta ,R_0)^{\tilde\chi}
 \left( 1+a \right)^{\hat{\chi}}
\end{eqnarray}
with constant $\Theta$  depending on $p,q,r,n,||h||_{L^r},R_0$ and positive exponents $\tilde\chi, \hat{\chi}$ depending on $p,q,r,n$. The conclusion follows taking the limit as $a\to ||u||_{L^\infty}.$
\end{proof}

We are now in position,  using the Caccioppoli inequality of Lemma \ref{Caccio}, to give the

\begin{proof}[Proof of Theorem \ref{highdiff}] Using \eqref{stimaquasifin4} in the right hand side in \eqref{Caccioppoli} with $\gamma=0$, we have
\begin{eqnarray*}
 && \int_{B_{\frac{R_0}{2}}}   |D^2 {\uj}|^2 \, dx\le \int_{B_{\frac{R_0}{2}}}  (1 + |D\uj|^2)^{\frac{p-2}{2}} |D^2 {\uj}|^2 \, dx\cr\cr
&\le & C ||V(Du_j)||^{2q-p}_{\LL^\infty(B_{R_0})}  \, \int_{B_{R_0}}   h^2(x) \, dx+\frac{C}{R_0^2} ||V(Du_j)||^{q}_{\LL^\infty(B_{R_0})}\cr\cr 
&\le&   C\left( 1+||u||_{\LL^{\infty}(B_{R_0}; \mathbb{R}^N}) \right)^{\hat{\chi}},
\end{eqnarray*}
where $C\equiv C (n,N,  \nu, \tilde L, ||h||_{L^r(\Omega)}, R_0)$. The conclusion now easily follows taking the limit as $j\to+\infty$ in the previous estimate and recalling that $u_j\to u$ in $\WW^{1,p}_{\mathrm{loc}}(\Omega; \mathbb{R}^N)$.
\end{proof}
We conclude mentioning that the same argument leads to the following second order regularity result
$$\int_{B_{\frac{R_0}{2}}}  (1 + |D\uj|^2)^{\frac{p-2+\gamma}{2}} |D^2 {\uj}|^2 \, dx\le C_\gamma\left( 1+||u||_{\LL^{\infty}(B_{R_0}; \mathbb{R}^N}) \right)^{{\chi_\gamma}}, $$
where now both the constant and the exponent depend on $\gamma$.

\bigskip

\noindent {\bf Data availability}

\vspace{3mm}

\noindent No datasets were generated or analysed during the current study.

\bigskip

\noindent {\bf Acknowledgments.}
M. Eleuteri and A. Passarelli di Napoli have been partially supported by the Gruppo
Nazionale per l’Analisi Matematica, la Probabilità e le loro Applicazioni (GNAMPA) of the Istituto
Nazionale di Alta Matematica (INdAM). Moreover M. Eleuteri has been partially supported by PRIN 2020 ``Mathematics for industry 4.0 (Math4I4)'' (coordinator P. Ciarletta) while
A. Passarelli di Napoli has been partially supported by Università
degli Studi di Napoli Federico II through the Project FRA ( 000022-75-2021-FRA-PASSARELLI) and by the Sustainable Mobility Center (Centro Nazionale per la Mobilità Sostenibile – CNMS)  Spoke 10 Logistica Merci. 
\\
The authors wish to thank Prof. F. Leonetti for useful suggestions and comments.


\begin{thebibliography}{99}



\bibitem{Adimurthy}\textsc{K. Adimurthi , V. Tewary,} \emph{Borderline Lipschitz regularity for bounded minimizers of functionals with $(p,q)$-growth} ArXiv Preprint.

\bibitem{BM} \textsc{L. Beck, G. Mingione,} {\it Lipschitz Bounds and Nonuniform Ellipticity,} Comm. Pure Appl. Math., {\bf 73}, (5), (2020) 944-1034.

\bibitem{BF05} \textsc{M. Bildhauer, M. Fuchs,} {\it $\mathcal{C}^{1, \alpha}$ solutions to non-autonomous anisotropic variational problems,} Calc. Var. \& PDE, {\bf 24} (2005), 309–340.

%
\bibitem{BMS18} \textsc{V. B\"ogelein, P. Marcellini, C. Scheven,} \emph{A variational approach to doubly nonlinear
equations,} Rend. Lincei, Mat. Appl., {\bf 29} (2018), 739-772.
%

\bibitem{BS2022} \textsc{P. Bella, M. Sch\"affner,}
{\it Lipschitz bounds for integral functionals with $(p,q)$-growth conditions}, Adv. Calc.  Var., (2022), to appear.


\bibitem{CKP}\textsc{M Carozza, J Kristensen, A. Passarelli di Napoli,} \emph{ Higher differentiability of minimizers of convex variational integrals,} Ann. Inst. H. Poincar\'e -- Anal. Non Lin\'{e}aire  {\bf 28} (2011), 395--411

  %
 
 



%


%
\bibitem{Choe}  \textsc{H. J. Choe,} \emph{Interior behavior of minimizers for certain functionals with non standard growth conditions}, Nonlinear Analysis TMA, \textbf{19} (1992), 933-945.
%
%
%
%

\bibitem{colmin}\textsc{M. Colombo, G. Mingione,} {\it Regularity for double phase variational problems,} Arch. Rat. Mech. Anal., {\bf 215} (2), (2015) 443-496.

\bibitem{colmin2}\textsc{M. Colombo, G. Mingione,} {\it  Bounded minimisers of double phase variational integrals,} Arch. Rat. Mech. Anal., {\bf 218} (1), (2015) 219-273.

\bibitem{CGGP}\textsc{G. Cupini, F. Giannetti, R. Giova, A.  Passarelli di Napoli,} \emph{\ Regularity results for vectorial minimizers of a class of degenerate convex integrals}, J. Differential Equations, \text{\ 265}{\ (2018) 4375-4416}

\bibitem{CupGuiMas} \textsc{G. Cupini, M. Guidorzi, E.  Mascolo,} \emph{\ Regularity of
minimizers of vectorial integrals with $p-q$ growth}, Nonlinear Anal., \textbf{54}{\ (2003) 591-616}.

\bibitem{CMM14} {\sc G. Cupini, P. Marcellini, E. Mascolo,} {\it Existence and regularity for elliptic equations under $p, q$-growth}, Adv. Diff. Equ., {\bf 19} (2014), 693–724.

\bibitem{CMMPdN} \textsc{G. Cupini, P. Marcellini, E. Mascolo, A. Passarelli di Napoli}, {\it Lipschitz regularity for degenerate elliptic integrals with $p,q$-growth}, Adv. Calc. Var., {\bf 16} (2) (2023), 443–465.

\bibitem{DFM2021} \textsc{C. De Filippis, G. Mingione}, {\it Lipschitz bounds and nonautonomous integrals,} Arch. Ration. Mech. Anal., {\bf 242} (2021) 973-1057.







%
%
%
%
%
%
\bibitem{DeFilLeo} \textsc{F. De Filippis, F. Leonetti,} \emph{No Lavrentiev gap for some double phase integrals
}, Adv. Calc. Var. (2022), https://doi.org/10.1515/acv-2021-0109.



\bibitem{EMM} \textsc{M. Eleuteri, P. Marcellini, E. Mascolo}, {\it Lipschitz estimates for systems with ellipticity conditions at
infinity,} Ann. Mat. Pura Appl., \textbf{195}, (2016) 1575–1603.

\bibitem{EMM20} \textsc{M. Eleuteri, P. Marcellini, E. Mascolo,} {\it Regularity
for scalar integrals without structure conditions}, Adv. Calc. Var., 
\textbf{13} (2020), 279-300.

\bibitem{EMMP} \textsc{M. Eleuteri, P. Marcellini, E. Mascolo, S. Perrotta,} {\it Local Lipschitz continuity for energy integrals with slow growth,} Annali Mat. Pura Appl., \textbf{201} (3), (2022) 1005–1032.

\bibitem{EPdN23} \textsc{M. Eleuteri, A. Passarelli di Napoli}, {\it Lipschitz regularity of minimizers of variational integrals with variable exponents,} Nonlinear Analysis: Real World Applications, (2023), 71, 103815

\bibitem{ELM99} {\sc L. Esposito, F. Leonetti, G. Mingione,} {\it Regularity for minimisers of functionals
with $p-q$ growth,} Nonlinear Diff. Equ. Appl., {\bf 6} (1999), 133–148.

%
%
%
%
%
%
%

%
\bibitem{GenGioTor} \textsc{A. Gentile, R. Giova, A. Torricelli,} \emph{Regularity results for bounded solutions to obstacle problems with non standard growth}, Mediterranean Journal of Mathematics, {\bf 19} (6), (2022) 270.
%
\bibitem{GPdN} {\sc R. Giova, A. Passarelli di Napoli,} \emph{Regularity results for a priori bounded minimizers of non-autonomous functionals with discontinuous coefficients}, Adv. Calc. Var., (to appear). DOI: https://doi.org/10.1515/acv-2016-0059
%
%
\bibitem{Giusti} {\sc E. Giusti,} \emph{Direct methods in the calculus of variations}. World scientific publishing Co. (2003).

\bibitem{AGGIpo} \textsc{A. G. Grimaldi, E. Ipocoana,} \emph{Higher differentiability to a class of obstacle problems with (p,q)- growth}, Forum Math., (2023) to appear.

\bibitem{L94} {\sc G.M. Lieberman,} {\it Gradient estimates for a new class of degenerate elliptic and parabolic equations,} Ann. Scuola Norm. Sup. Pisa Cl. Sci. (IV), {\bf 21} (1994), 497–522.

%
%
%
%
\bibitem{M89} \textsc{P. Marcellini,} \emph{Regularity of minimizers of integrals of the calculus of variations with nonstandard growth conditions}, Arch. Ration. Mech. Anal., {\bf 105} (1989), no. 3, 267--284.
%
\bibitem{M91} \textsc{P. Marcellini,} \emph{Regularity and existence of solutions of elliptic equations with $p, q$-growth conditions}, J. Differential Equations, {\bf 90} (1991), no. 1, 1--30.
%

\bibitem{HO} \textsc{P. H\"ast\"o, J. Ok,} {\it Regularity theory for non-autonomous problems with a priori assumptions,} Calc. Var. PDEs, (2023) to appear. Preprint arXiv:2209.08917

\bibitem{M2021} \textsc{P. Marcellini,} {\it Growth conditions and regularity for weak solutions to nonlinear elliptic pdes,} J. Math. Anal. Appl., J. Math. Anal. Appl., 
{\bf 501}, (1), (2021), 124408. 
 
\bibitem{M2022} \textsc{P. Marcellini,} {\it Local Lipschitz continuity for $p, q$-PDEs
with explicit $u$-dependence,} Nonlinear Analysis {\bf 226}, (2023), 113066.

\bibitem{PdN14-1} \textsc{A. Passarelli di Napoli,} \emph{Higher differentiability of minimizers of variational integrals with Sobolev coefficients}, Adv. Calc. Var., {\bf 7}, (1), (2014), 59-89.

\bibitem{PdN14-2} \textsc{A. Passarelli di Napoli,} \emph{Higher differentiability of solutions of elliptic systems with Sobolev coefficients: the case $p = n = 2$}, Pot. Anal., {\bf 41}, (3), (2014), 715-735.


\bibitem{Z} \textsc{V.V. Zhikov,}{\it Averaging of functionals of the calculus of variations and elasticity theory,} Izv. Akad. Nauk SSSR Ser. Mat., \textbf{\ 50}{\ (1986) 675-710}.


\bibitem{ZkO}\textsc{V.V. Zhikov, S. M. Kozlov S. M.,  O. A. Oleinik,}{\ Homogenization of differential operators and integral functionals, } Springer-Verlag, Berlin {\ (1994)}.


\end{thebibliography}
\end{document}